\documentclass[secthm,seceqn,amsthm,ussrhead,12pt]{amsart}
\usepackage[utf8]{inputenc}
\usepackage[english]{babel}
\usepackage{amssymb,amsmath,amsthm,amsfonts,xcolor,enumerate,hyperref,comment,longtable,cleveref}

\usepackage{times}
\usepackage{cite}
\usepackage{pdflscape}
\usepackage{ulem}
\usepackage[mathcal]{euscript}
\usepackage{tikz}
\usepackage{hyperref}
\usepackage{cancel}
\usepackage{stmaryrd}
\usepackage{ulem}
\usepackage{multicol}

\usetikzlibrary{arrows}

\newtheorem{theorem}{Theorem}
\newtheorem{lemma}[theorem]{Lemma}

\newtheorem{remark}[theorem]{Remark}

\newtheorem{definition}[theorem]{Definition}

\newenvironment{Proof}[1][Proof.]{\begin{trivlist}
\item[\hskip \labelsep {\bfseries #1}]}{\flushright
$\Box$\end{trivlist}}

\usepackage{stmaryrd}
\usepackage{xcolor}

\begin{document}
\noindent{\Large 
Central extensions of $3$-dimensional Zinbiel algebras}
\footnote{
This work was started during the research stay of I. Kaygorodov at the Department of Mathematics, parcially funded by the {\it Coloquio de Matem\'atica} (CR 4430) of the University of Antofagasta. 
The work was supported by 
RFBR 20-01-00030;  FAPESP 18/15627-2, 19/03655-4; CNPq 302980/2019-9;
AP08052405 of MES RK. \\
 
}
 
   \medskip

   {\bf  
   Mar\'ia Alejandra Alvarez$^{a}$, 
   Thiago Castilho de Mello$^{b}$ \&
    Ivan Kaygorodov$^{c}$ 
   
    \medskip
}

{\tiny

$^{a}$ Departamento de Matem\'aticas, Facultad de Ciencias B\'asicas, Universidad de Antofagasta, Chile.

$^{b}$ Instituto de Ciência e Tecnologia, Universidade Federal de São Paulo, São José dos Campos, Brasil. 

$^{c}$ CMCC, Universidade Federal do ABC. Santo Andr\'e, Brasil.

\

   \medskip

   E-mail addresses:

\smallskip   
    Mar\'ia Alejandra Alvarez (maria.alvarez@uantof.cl),

\smallskip    
Thiago Castillo de Mello (tcmello@unifesp.br),

\smallskip
    Ivan Kaygorodov (kaygorodov.ivan@gmail.com).

}

\medskip

\noindent {\bf Abstract:}
We  describe all central extensions of all $3$-dimensional non-zero complex Zinbiel algebras.
As a corollary, we have a full classification of $4$-dimensional non-trivial complex Zinbiel algebras
and a full classification of 
 $5$-dimensional non-trivial complex Zinbiel algebras with $2$-dimensional annihilator,
 which gives the principal step in the algebraic classification of $5$-dimensional Zinbiel algebras.

   \medskip


   \medskip

\noindent{}{\bf 2020 MSC:} 17A30

   \medskip

   \medskip
   \medskip

\section*{Introduction}
	Algebraic classification (up to isomorphism) of $n$-dimensional algebras from a certain variety
	defined by some family of polynomial identities is a classical problem in the theory of non-associative algebras.
	There are many results related to algebraic classification of small dimensional algebras in varieties of 
	Jordan, Lie, Leibniz, Zinbiel and many other algebras \cite{ack, cfk20, kkp20, degr3, usefi1, degr1, degr2, ha16, krs19, klp20, kks19, kkl19, hac18}.
	
	An algebra $\bf A$ is called a {\it Zinbiel algebra} if it satisfies the identity 
	$(xy)z=x(yz+zy).$
	Zinbiel algebras were introduced by Loday in \cite{loday} and studied in \cite{abash,adashev19, cam13,mukh, dok,  anau, saha, dzhuma, ckkk20, dzhuma5, dzhuma19, kppv, ualbay, yau}. Under the Koszul duality, the operad of Zinbiel algebras is dual to the operad of Leibniz algebras. Zinbiel algebras are also related to Tortkara algebras \cite{dzhuma} and Tortkara triple systems \cite{brem}. 
	More precisely, every Zinbiel algebra with the commutator multiplication gives a Tortkara algebra (also about Tortkara algebras, see, \cite{gkp,gkk}). 
	Tortkara algebras have recently sprung up in unexpected areas of mathematics \cite{tortnew1,tortnew2}.

	Central extensions play an important role in quantum mechanics: one of the earlier
	encounters is by means of Wigner's theorem which states that a symmetry of a quantum
	mechanical system determines an (anti-)unitary transformation of a Hilbert space.
	Another area of physics where one encounters central extensions is the quantum theory
	of conserved currents of a Lagrangian. These currents span an algebra which is closely
	related to so called affine Kac-Moody algebras, which are universal central extensions
	of loop algebras.
	Central extensions are needed in physics, because the symmetry group of a quantized
	system usually is a central extension of the classical symmetry group, and in the same way
	the corresponding symmetry Lie algebra of the quantum system is, in general, a central
	extension of the classical symmetry algebra. Kac-Moody algebras have been conjectured
	to be a symmetry group of a unified superstring theory. The centrally extended Lie
	algebras play a dominant role in quantum field theory, particularly in conformal field
	theory, string theory and in $M$-theory.
	In the theory of Lie groups, Lie algebras and their representations, a Lie algebra extension
	is an enlargement of a given Lie algebra $g$ by another Lie algebra $h.$ Extensions
	arise in several ways. There is a trivial extension obtained by taking a direct sum of
	two Lie algebras. Other types are a split extension and a central extension. Extensions may
	arise naturally, for instance, when forming a Lie algebra from projective group representations.
	A central extension and an extension by a derivation of a polynomial loop algebra
	over finite-dimensional simple Lie algebra gives a Lie algebra which is isomorphic to a
	non-twisted affine Kac-Moody algebra \cite[Chapter 19]{bkk}. Using the centrally extended loop
	algebra one may construct a current algebra in two spacetime dimensions. The Virasoro
	algebra is the universal central extension of the Witt algebra, the Heisenberg algebra is
	the central extension of a commutative Lie algebra  \cite[Chapter 18]{bkk}.
	
	The algebraic study of central extensions of Lie and non-Lie algebras has a very long history \cite{omirov,ha17,ha16nm,hac16,kkl18,ss78,cam19,cam20, zusmanovich}.
	For example, all central extensions of some filiform Leibniz algebras were classified in \cite{omirov}
	and all central extensions of filiform associative algebras were classified in \cite{kkl18}.
	Skjelbred and Sund used central extensions of Lie algebras for a classification of low dimensional nilpotent Lie algebras  \cite{ss78}.
	After that, the method introduced by Skjelbred and Sund was used to 
	describe all non-Lie central extensions of all $4$-dimensional Malcev algebras \cite{hac16},
	all non-associative central extensions of $3$-dimensional Jordan algebras \cite{ha17},
	all anticommutative central extensions of $3$-dimensional anticommutative algebras \cite{cfk182},
	all central extensions   of $2$-dimensional algebras \cite{cfk18}.
	Note that the method of central extensions is an important tool in the classification of nilpotent algebras.
	It was used to describe
	all $4$-dimensional nilpotent associative algebras \cite{degr1},
	all $4$-dimensional nilpotent assosymmetric  algebras \cite{ikm19},
	all $4$-dimensional nilpotent bicommutative algebras \cite{kpv19},
	all $4$-dimensional nilpotent Novikov  algebras \cite{kkk18},
	all $4$-dimensional commutative algebras \cite{fkkv19},
	all $5$-dimensional nilpotent Jordan algebras \cite{ha16},
	all $5$-dimensional nilpotent restricted Lie algebras \cite{usefi1},
	all $5$-dimensional anticommutative algebras \cite{fkkv19},
	all $6$-dimensional nilpotent Lie algebras \cite{degr3,degr2},
	all $6$-dimensional nilpotent Malcev algebras \cite{hac18},
	all $6$-dimensional nilpotent binary Lie algebras\cite{ack},  
	all $6$-dimensional nilpotent anticommutative $\mathfrak{CD}$-algebras \cite{ack},
	all $6$-dimensional nilpotent Tortkara algebras\cite{gkk},  
	and some other.

\ 

\paragraph{\bf Motivation and contextualization} 
Given algebras ${\bf A}$ and ${\bf B}$ in the same variety, we write ${\bf A}\to {\bf B}$ and say that ${\bf A}$ {\it degenerates} to ${\bf B}$, or that ${\bf A}$ is a {\it deformation} of ${\bf B}$, if ${\bf B}$ is in the Zariski closure of the orbit of ${\bf A}$ (under the aforementioned base-change action of the general linear group). The study of degenerations of algebras is very rich and closely related to deformation theory, in the sense of Gerstenhaber. It offers an insightful geometric perspective on the subject and has been the object of a lot of research.
In particular, there are many results concerning degenerations of algebras of small dimensions in a  variety defined by a set of identities.
One of the main problems of the {\it geometric classification} of a variety of algebras is a description of its irreducible components. In the case of finitely-many orbits (i.e., isomorphism classes), the irreducible components are determined by the rigid algebras --- algebras whose orbit closure is an irreducible component of the variety under consideration. 

Observe that the algebraic classification of central extensions of $3$-dimensional Zinbiel algebras gives 
the description of all Zinbiel algebras with $(n-3)$-dimensional annihilator (thanks to \cite{dzhuma5} we know that every finite-dimensional Zinbiel algebra is nilpotent), 
which is a principal step in the classification of all $n$-dimensional Zinbiel algebras.
However, an algebraic classification of $n$-dimensional Zinbiel algebras gives a way for obtaining the geometric classification of  all
complex $n$-dimensional Zinbiel algebras.

\section{Preliminaries}
\subsection{Previous definitions and methods}
During this paper, we are using the notations and methods well written in \cite{ha17,hac16,cfk18}
and adapted for the Zinbiel case with some modifications. 
From now, we will give only some important definitions.

Let $({\bf A}, \cdot)$ be a Zinbiel algebra over  $\mathbb{C}$  and let $\mathbb V$ be a vector space over $\mathbb{C}.$ Then the $\mathbb{C}$-linear space ${\rm Z}^2\left(
\bf A,\mathbb V \right) $ is defined as the set of all  bilinear maps $%
\theta :{\bf A} \times {\bf A} \longrightarrow {\mathbb V},$
such that \[\theta(xy,z)=\theta(x,yz+zy).\]
 Its elements will be called \textit{cocycles}. For a
linear map $f$ from $\bf A$ to  $\mathbb V$, if we write $\delta f\colon {\bf A} \times
{\bf A} \longrightarrow {\mathbb V}$ by $\delta f \left( x,y \right) =f(xy )$, then $%
\delta f\in {\rm Z}^2\left( {\bf A},{\mathbb V} \right) $. We define ${\rm B}^2\left(
{\bf A},{\mathbb V}\right) =\left\{ \theta =\delta f\ :f\in {\rm Hom}\left( {\bf A},{\mathbb V}\right) \right\}$.
One can easily check that ${\rm B}^2(\bf A,\mathbb V)$ is a linear subspace of $%
{\rm Z}^2\left( {\bf A},{\mathbb V}\right) $ whose elements are called \textit{%
coboundaries}. We define the \textit{second cohomology space} $%
{\rm H}^2\left( {\bf A},{\mathbb V}\right) $ as the quotient space ${\rm Z}^2%
\left( {\bf A},{\mathbb V}\right) \big/{\rm B}^2\left( {\bf A},{\mathbb V}\right).$

Let ${\rm Aut}\left( {\bf A} \right) $ be the automorphism group of the Zinbiel  algebra ${\bf A} $ and let $\phi \in {\rm Aut}\left( {\bf A}\right) $. For $\theta \in
{\rm Z}^2\left( {\bf A},{\mathbb V}\right) $ define $\phi \theta \left( x,y\right)
=\theta \left( \phi \left( x\right) ,\phi \left( y\right) \right) $. Then $%
\phi \theta \in {\rm Z}^2\left( {\bf A},{\mathbb V}\right) $. So, ${\rm Aut}\left( {\bf A}\right)$
acts on ${\rm Z}^2\left( {\bf A},{\mathbb V}\right) $. It is easy to verify that $%
{\rm B}^2\left( {\bf A},{\mathbb V}\right) $ is invariant under the action of ${\rm Aut}\left(
{\bf A}\right) $ and then we have that ${\rm Aut}\left( {\bf A}\right) $ acts on ${\rm H}^2\left( {\bf A},{\mathbb V}\right) $.

Let $\bf A$ be a Zinbiel  algebra of dimension $m<n$ over $\mathbb{C},$ 
and let ${\mathbb V}$ be a $\mathbb{C}$-vector
space of dimension $n-m$. For any $\theta \in {\rm Z}^2\left({\bf A},{\mathbb V}\right) $ define on the linear space ${\bf A}_{\theta }:={\bf A}\oplus {\mathbb V}$ the
bilinear product \textquotedblleft\ $\left[ -,-\right] _{{\bf A}_{\theta }}$" by $
\left[ x+x^{\prime },y+y^{\prime }\right] _{{\bf A}_{\theta }}=
 xy +\theta \left( x,y\right) $ for all $x,y\in {\bf A},x^{\prime },y^{\prime }\in {\mathbb V}$.
The algebra ${\bf A}_{\theta}$ is a Zinbiel algebra which is called an $(n-m)$%
{\it{-dimensional central extension}} of ${\bf A}$ by ${\mathbb V}$. Indeed, we have, in a straightforward way, that ${\bf A_{\theta}}$ is a Zinbiel algebra if and only if $\theta \in {\rm Z}^2({\bf A}, \mathbb{C})$.

We call the
set ${\rm Ann}(\theta)=\left\{ x\in {\bf A}:\theta \left( x, {\bf A} \right)+ \theta \left({\bf A} ,x\right) =0\right\} $
the {\it{annihilator}} of $\theta $. We recall that the {\it{annihilator}} of an  algebra ${\bf A}$ is defined as
the ideal ${\rm Ann}\left( {\bf A} \right) =\left\{ x\in {\bf A}:  x{\bf A}+ {\bf A}x =0\right\}$ and observe
 that
${\rm Ann}\left( {\bf A}_{\theta }\right) = {\rm Ann}(\theta) \cap {\rm Ann}\left( {\bf A}\right)
 \oplus {\mathbb V}.$

\medskip

We have the next  key result:

\begin{lemma}
Let ${\bf A}$ be an $n$-dimensional Zinbiel algebra such that $\dim({\rm Ann}({\bf A}))=m\neq0$. Then there exist, up to isomorphism, a unique $(n-m)$-dimensional Zinbiel algebra ${\bf A}'$ and a bilinear map $\theta \in {\rm Z}^2({\bf A}, {\mathbb V})$ with ${\rm Ann}({\bf A})\cap {\rm Ann}(\theta)=0$, where $\mathbb V$ is a vector space of dimension m, such that ${\bf A}\cong {\bf A}'_{\theta}$ and $
{\bf A}/{\rm Ann}({\bf A})\cong {\bf A}'$.
\end{lemma}

\begin{Proof}

Let ${\bf A}'$ be a linear complement of ${\rm Ann}({\bf A})$ in ${\bf A}$. Define a linear map $P: {\bf A} \longrightarrow {\bf A}'$ by $P(x+v)=x$ for $x\in {\bf A}'$ and $v\in {\rm Ann}({\bf A})$ and define a multiplication on ${\bf A}'$ by $[x, y]_{{\bf A}'}=P(x y)$ for $x, y \in {\bf A}'$.
For $x, y \in {\bf A}$ then
$$P(xy)=P((x-P(x)+P(x))(y- P(y)+P(y)))=P(P(x) P(y))=[P(x), P(y)]_{{\bf A}'}$$

Since $P$ is a homomorphism then $P({\bf A})={\bf A}'$ is a Zinbiel algebra and $
{\bf A}/{\rm Ann}({\bf A})\cong {\bf A}'$, which give us the uniqueness. Now, define the map $\theta: {\bf A}'\times{\bf A}'\longrightarrow {\rm Ann}({\bf A})$ by $\theta(x,y)=xy- [x,y]_{{\bf A}'}$. Thus, ${\bf A}'_{\theta}$ is ${\bf A}$ and therefore $\theta \in {\rm Z}^2({\bf A}, {\mathbb V})$ and ${\rm Ann}({\bf A})\cap {\rm Ann}(\theta)=0$.
\end{Proof}

\bigskip

However, in order to solve the isomorphism problem we need to study the
action of ${\rm Aut}\left( {\bf A}\right) $ on ${\rm H}^2\left( {\bf A},\mathbb{C}\right).$ To do that, let us fix a basis for ${\mathbb V}$,  $e_{1},\ldots ,e_{s}$ , and let $
\theta \in {\rm Z}^2\left( {\bf A},{\mathbb V}\right) $. Then $\theta$ can be uniquely
written as $\theta \left( x,y\right) =\underset{i=1}{\overset{s}{\sum }}%
\theta _{i}\left( x,y\right) e_{i}$, where $\theta _{i}\in
{\rm Z}^2\left( {\bf A},\mathbb{C}\right) $. Moreover, ${\rm Ann}(\theta)={\rm Ann}(\theta _{1})\cap {\rm Ann}(\theta _{2})\cap \ldots \cap {\rm Ann}(\theta _{s})$. Further, $\theta \in
{\rm B}^2\left( {\bf A},{\mathbb V}\right) $\ if and only if all $\theta _{i}\in {\rm B}^2\left( {\bf A},%
\mathbb{C}\right) $.

\bigskip

\begin{definition}
Let ${\bf A}$ be an algebra and let $I$ be a subspace of $\operatorname{Ann}({\bf A})$. If ${\bf A}={\bf A}_0 \oplus I$
then $I$ is called an {\it annihilator component} of ${\bf A}$.
\end{definition}
\begin{definition}
A central extension of an algebra $\bf A$ without annihilator component is called a {\it non-split central extension}.
\end{definition}

It is not difficult to prove, (see \cite[%
Lemma 13]{hac16}), that given a Zinbiel algebra ${\bf A}_{\theta}$, if we write as
above $\theta \left( x,y\right) =\underset{i=1}{\overset{s}{\sum }}$
 $\theta_{i}\left( x,y\right) e_{i}\in {\rm Z}^2\left( {\bf A},{\mathbb V}\right) $ and we have
${\rm Ann}(\theta)\cap {\rm Ann}\left( {\bf A}\right) =0$, then ${\bf A}_{\theta }$ has an
annihilator component if and only if $\left[ \theta _{1}\right] ,\left[
\theta _{2}\right] ,\ldots ,\left[ \theta _{s}\right] $ are linearly
dependent in ${\rm H}^2\left( {\bf A}, \mathbb{C} \right) $.

\bigskip

Let ${\mathbb V}$ be a finite-dimensional vector space over $\mathbb{C}.$ The {\it{%
Grassmannian}} $G_{k}\left( {\mathbb V}\right) $ is the set of all $k$-dimensional
linear subspaces of $ {\mathbb V}$. Let $G_{s}\left( {\rm H}^2\left( {\bf A},\mathbb{C}%
\right) \right) $ be the Grassmannian of subspaces of dimension $s$ in $%
{\rm H}^2\left( {\bf A},\mathbb{C} \right) $. There is a natural action of $%
{\rm Aut}\left( {\bf A}\right) $ on $G_{s}\left( {\rm H}^2\left( {\bf A},\mathbb{C} 
\right) \right) $. Let $\phi \in {\rm Aut}\left( {\bf A}\right) $. For $W=\left\langle %
\left[ \theta _{1}\right] ,\left[ \theta _{2}\right] , \ldots ,\left[ \theta _{s}%
\right] \right\rangle \in G_{s}\left( {\rm H}^2\left( {\bf A}, \mathbb{C} %
\right) \right) $ define $\phi W=\left\langle \left[ \phi \theta _{1}\right]
,\left[ \phi \theta _{2}\right] , \ldots,\left[ \phi \theta _{s}\right]
\right\rangle.$ Then $\phi W\in G_{s}\left({\rm H}^2\left( {\bf A},\mathbb{C}  \right) \right) $. We denote the orbit of $W\in G_{s}\left(
{\rm H}^2\left( {\bf A},\mathbb{C} \right) \right) $ under the action of $%
{\rm Aut}\left( {\bf A}\right) $ by ${\rm Orb}\left( W\right) $. Since given
\begin{equation*}
W_{1}=\left\langle \left[ \theta _{1}\right] ,\left[ \theta _{2}\right] ,\ldots ,\left[ \theta_{s}\right] \right\rangle,
W_{2}=\left\langle \left[ \vartheta_{1}\right] ,\left[ \vartheta _{2}\right] ,\ldots ,\left[ \vartheta _{s}\right]
\right\rangle \in G_{s}\left( {\rm H}^2\left( {\bf A},\mathbb{C}\right)
\right)
\end{equation*}%
we easily have that in case $W_{1}=W_{2}$, then $\underset{i=1}{\overset{s}{%
\cap }}{\rm Ann}(\theta _{i})\cap {\rm Ann}\left( {\bf A}\right) =\underset{i=1}{\overset{s}%
{\cap }}{\rm Ann}(\vartheta _{i})\cap {\rm Ann}\left( {\bf A}\right) $, we can introduce
the set

\begin{equation*}
T_{s}\left( {\bf A}\right) =\left\{ 
W=\left\langle \left[ \theta _{1}\right] ,\left[ \theta _{2}\right] ,\ldots ,\left[ \theta _{s}\right] \right\rangle \in
G_{s}\left( {\rm H}^2\left( {\bf A}, \mathbb{C} \right) \right) :\underset{i=1}{%
\overset{s}{\cap }} {\rm Ann}(\theta _{i})\cap {\rm Ann}\left( {\bf A}\right) =0\right\},
\end{equation*}
which is stable under the action of ${\rm Aut}\left( {\bf A}\right).$

\medskip

Now, let ${\mathbb V}$ be an $s$-dimensional linear space and let us denote by $%
{\rm E}\left( {\bf A},{\mathbb V}\right) $ the set of all non-split $s${\it{-}dimensional} central extensions of ${\bf A}$ by
${\mathbb V}.$ We can write
\begin{equation*}
{\rm E}\left( {\bf A},{\mathbb V}\right) =
\left\{ {\bf A}_{\theta }:\theta \left( x,y\right) =
\underset{i=1}{\overset{s}{\sum }}\theta _{i}\left( x,y\right) e_{i}\mbox{
and }\left\langle \left[ \theta _{1}\right] ,\left[ \theta _{2}\right] ,\ldots ,
\left[ \theta _{s}\right] \right\rangle \in T_{s}\left( {\bf A}\right) \right\} .
\end{equation*}%
Also we have the next result, which can be proved as \cite[Lemma 17]{hac16}.

\begin{lemma}
 Let ${\bf A}_{\theta },{\bf A}_{\vartheta }\in {\rm E}\left( {\bf A},{\mathbb V}\right).$ 
 Suppose that $\theta \left( x,y\right) =\underset{i=1}{\overset{s}{\sum }}%
\theta _{i}\left( x,y\right) e_{i}$ and $\vartheta \left( x,y\right) =%
\underset{i=1}{\overset{s}{\sum }}\vartheta _{i}\left( x,y\right) e_{i}$.
Then the Zinbiel algebras ${\bf A}_{\theta }$ and ${\bf A}_{\vartheta } $ are isomorphic
if and only if ${\rm Orb}\left\langle \left[ \theta _{1}\right] ,%
\left[ \theta _{2}\right],\ldots,\left[ \theta _{s}\right] \right\rangle =%
{\rm Orb}\left\langle \left[ \vartheta _{1}\right] ,\left[ \vartheta
_{2}\right] ,\ldots,\left[ \vartheta _{s}\right] \right\rangle $.
\end{lemma}

From here, there exists a one-to-one correspondence between the set of ${\rm Aut}
\left( {\bf A}\right) $-orbits on $T_{s}\left( {\bf A}\right) $ and the set of
isomorphism classes of ${\rm E}\left( {\bf A},{\mathbb V}\right).$ Consequently we have a
procedure that allows us, given the Zinbiel algebras ${\bf A}^{\prime }$ of
dimension $n-s$, to construct all non-split central extensions of ${\bf A}^{\prime }.$ This procedure would be:

\medskip

{\centerline{\it Procedure}}

\begin{enumerate}
\item For a given Zinbiel algebra ${\bf A}^{\prime }$ of dimension $%
n-s $, determine ${\rm H}^2( {\bf A}^{\prime },\mathbb{C}) $, ${\rm Ann}( {\bf A}^{\prime })
$ and ${\rm  Aut}( {\bf A}^{\prime }) $.

\item Determine the set of ${\rm Aut}( {\bf A}^{\prime }) $-orbits on $T_{s}(
{\bf A}^{\prime }) $.

\item For each orbit, construct the Zinbiel algebra corresponding to a
representative of it.
\end{enumerate}

\medskip

Finally, let us introduce some of notation. Let ${\bf A}$ be a Zinbiel algebra with
a basis $e_{1},e_{2}, \ldots,e_{n}$. Then by $\Delta _{ij}$\ we will denote the
Zinbiel bilinear form
$\Delta _{ij}:{\bf A}\times {\bf A}\longrightarrow \mathbb{C}$
with $\Delta _{ij}\left( e_{l},e_{m}\right) =\delta_{il}\delta_{jm}.$ 
Then the set $\left\{ \Delta
_{ij}:1\leq i, j\leq n\right\} $ is a basis for the linear space of
bilinear forms on ${\bf A}$. Then every $\theta \in
{\rm Z}^2\left( {\bf A},\mathbb{C}\right) $ can be uniquely written as $%
\theta =\underset{1\leq i,j\leq n}{\sum }c_{ij}\Delta _{{i}{j}}$, where $%
c_{ij}\in \mathbb{C}$.
Let us fix the following notation

\begin{longtable}{lcl}
$[{\bf A}]^i_j$  &  --- & $j$th $i$-dimensional central extension of   ${\bf A}.$ 
\end{longtable}

\subsection{Classification of Zinbiel algebras}
Thanks to \cite{dzhuma5} we have that every finite-dimensional Zinbiel algebra is nilpotent.
There are only one $2$-dimensional non-zero nilpotent  Zinbiel algebra given by the multiplication table: $e_1^2=e_2$.
Thanks to \cite{cfk18}, 
as central extensions of $2$-dimensional Zinbiel algebras, 
we can take the classification of all non-split $3$-dimensional Zinbiel algebras.
Now we have the classification of all $3$-dimensional Zinbiel algebras and their cohomology spaces: 
 
\begin{longtable}{|l|lll|l|} 
\hline
Algebra &\multicolumn{3}{l|}{ Multiplication table} & Cohomology \\
\hline

$\mathfrak{Z}_1$ & $ e_1e_1 = e_2$ & $e_1e_2 =\frac{1}{2} e_3$ & $e_2e_1=e_3$ &$\left\langle 2[\Delta_{13}]+3[\Delta_{22}]+6[\Delta_{31}] \right \rangle$ \\ 
\hline

$\mathfrak{N}_1^{\mathbb{C}}$&  $ e_1e_1 = e_2$ && & 
$\left\langle [\Delta_{13}],\ [\Delta_{12}]+2[\Delta_{21}],  
\ [\Delta_{31}],\ [\Delta_{33}] \right\rangle$ \\ \hline

$\mathfrak{N}_1$ &$e_1e_2=  e_3$ &$ e_2e_1=-e_3$& &
$\left\langle [\Delta_{11}],\ [\Delta_{12}],\ [\Delta_{13}],  \ [\Delta_{22}],\ [\Delta_{23}] \right\rangle$\\

\hline

$\mathfrak{N}_2(\beta)$  & $e_1e_1=  e_3 $&$e_1e_2=e_3$&$e_2e_2=\beta e_3$&
$\left \langle[\Delta_{11}],\ [\Delta_{21}],\ [\Delta_{22}] \right\rangle$\\ \hline

$\mathfrak{N}_3$    & $e_1e_1= e_3$& $e_1e_2=e_3$ & $e_2e_1=e_3$ & 
$\left\langle [\Delta_{11}],\ [\Delta_{12}],\ [\Delta_{22}] \right\rangle$ \\

\hline

\end{longtable}

\begin{remark} All $n$-dimensional central extensions of $\mathfrak{N}_2(\beta)$ and $\mathfrak{N}_3$
are trivial.
They are split or have $(n+1)$-dimensional annihilator, which gives them as central extensions of a $2$-dimensional Zinbiel algebra found in \cite{cfk18}.
\end{remark}

\subsection{Central extensions of $\mathfrak{Z}_1$} It is easy to see that from $\dim({\rm H}^2(\mathfrak{Z}_1,\mathbb{C}))=1$ we have only one non-split central extension of
$\mathfrak{Z}_1.$
It is
$$\begin{array}{llllllll} 
[\mathfrak{Z}_1]^1_1 &:& e_1e_1=e_2, & e_1e_2=\textstyle\frac{1}{2}e_3, & e_1e_3=2e_4, & e_2e_1=e_3, &  e_2e_2=3e_4, & e_3e_1=6e_4.\\
\end{array}$$


\section{Central extensions of $\mathfrak{N}_1^{\mathbb{C}}$} 
The multiplication table of $3$-dimensional Zinbiel algebra $\mathfrak{N}_1^{\mathbb{C}}$ is given  by: 
 \[ e_1^2=e_2. \]
 The automorphism group of $\mathfrak{N}_1^{\mathbb{C}}$ consists of invertible matrices of the form
 \[\phi=\begin{pmatrix}
x & 0 & 0\\
z & x^2 & t\\
u & 0 & y
\end{pmatrix}. \]
The cohomology space of $\mathfrak{N}_1^{\mathbb{C}}$ is given  by: 
\[  \nabla_1=[\Delta_{12}]+2[\Delta_{21}], \nabla_2=[\Delta_{13}], \nabla_3=[\Delta_{31}], \nabla_4=[\Delta_{33}].
\]
Since
\[\begin{pmatrix}
x & 0 & 0\\
z & x^2 & t\\
u & 0 & y
\end{pmatrix}^t\begin{pmatrix}
0 & \alpha_1 & \alpha_2\\
2\alpha_1 & 0 & 0\\
\alpha_3 & 0 & \alpha_4
\end{pmatrix}\begin{pmatrix}
x & 0 & 0\\
z & x^2 & t\\
u & 0 & y
\end{pmatrix}=
 \begin{pmatrix}
\alpha^* & \alpha_1^* & \alpha_2^*\\
2\alpha_1^* & 0 & 0\\
\alpha_3^* & 0 & \alpha_4^*
\end{pmatrix},\]
where
\begin{longtable}{lcl}
$\alpha_1^*$&$=$&$x^3 \alpha_1$\\
$\alpha_2^*$&$=$&$t x \alpha_1 + x y \alpha_2 + u y \alpha_4$\\
$\alpha_3^*$&$=$&$2 t x \alpha_1 + x y \alpha_3 + u y \alpha_4$\\
$\alpha_4^*$&$=$&$y^2 \alpha_4$
\end{longtable}
we obtain that the action of $\operatorname{Aut}\left(\mathfrak{N}_1^{\mathbb{C}}\right)$ on a subspace $\langle \sum\limits_{i=1}^4 \alpha_i \nabla_i  \rangle$ is given by
$\langle \sum\limits_{i=1}^4 \alpha_i^* \nabla_i \rangle.$

\bigskip

\subsection{$1$-dimensional central extensions of $\mathfrak{N}_1^{\mathbb{C}}$}
Consider an element 
$ \theta_1= \sum\limits_{i=1}^4 \alpha_i \nabla_i.$ 
We are interested in elements with  $(\alpha_2, \alpha_3, \alpha_4) \neq (0,0,0)$
and $\alpha_1\neq 0.$
We provide the orbit of every possible case:
\begin{enumerate}

\item  $\alpha_4\neq 0, \alpha_1\neq 0:$
 by choosing 
        $x=\sqrt[3]{\frac{\alpha_4}{\alpha_1}},$  
        $y=1,$
        $z=0,$
        $t=\frac{\alpha_2-\alpha_3}{\alpha_1},$
        and
        $u=\frac{-2 \alpha_2+\alpha_3}{\sqrt[3]{\alpha_1  \alpha_4^2}}$
        we have the representative
            $\langle\nabla_1+\nabla_4\rangle.$

  \item  $\alpha_4= 0, \alpha_1 \neq 0, \alpha_3\neq 2 \alpha_2:$
 by choosing 
        $x=1,$  
        $y=-\frac{2 \alpha_1}{\alpha_3-2 \alpha_2},$
        $z=0,$
        $t=\frac{\alpha_3}{\alpha_3-2 \alpha_2},$
        and
        $u=0$
        we have the representative
            $\langle \nabla_1+\nabla_2\rangle.$             

    \item  $\alpha_4= 0, \alpha_1 \neq 0, \alpha_3=  2 \alpha_2:$
 by choosing 
        $x=1,$  
        $y=-\frac{ \alpha_1}{\alpha_2},$
        $z=0,$
        $t=1,$
        and
        $u=0$
        we have the representative
            $\langle \nabla_1 \rangle.$

\end{enumerate}

Summarizing and noting that $\langle \nabla_1\rangle$ gives a split algebra,
we have the following distinct nontrivial orbits:
\[  \langle \nabla_1+\nabla_4\rangle,  \ \langle \nabla_1+\nabla_2\rangle.  \ 
\]

\bigskip

\subsection{$2$-dimensional central extensions of $\mathfrak{N}_1^{\mathbb{C}}$}
Consider  elements 
$ \theta_1= \sum\limits_{i=1}^4 \alpha_i \nabla_i$
and 
$ \theta_2= \sum\limits_{i=1}^4 \beta_i \nabla_i.$ 
We are interested in elements with $(\alpha_2, \alpha_3, \alpha_4,\beta_2, \beta_3, \beta_4) \neq (0,0,0,0,0,0)$ and $(\alpha_1, \beta_1) \neq (0,0).$
We provide the orbit of every possible case:
\begin{enumerate}
\item  $\alpha_4=1, \beta_4=0.$ Then
\begin{enumerate} 
\item $ \beta_1=1, \alpha_1=0, \alpha_2\neq \alpha_3, 2\beta_2\neq \beta_3:$
 by choosing 
        $x=\frac{(\alpha_2-\alpha_3) (2 \beta_2-\beta_3)}{2},$  
        $y=\frac{(\alpha_2-\alpha_3)^2 (2 \beta_2-\beta_3)}{2},$
        $z=0,$
        $t=-\frac{\beta_3(\alpha_2-\alpha_3)^2 (2 \beta_2-\beta_3) }{4},$
        and
        $u=-\frac{\alpha_3(\alpha_2-\alpha_3)  (2 \beta_2-\beta_3)}{2}$
        we have the representative
            $\langle \nabla_2+\nabla_4, \nabla_1+\nabla_2\rangle.$     
      
\item $ \beta_1=1, \alpha_1=0, \alpha_2\neq \alpha_3, 2\beta_2= \beta_3:$
 by choosing 
        $x=1,$  
        $y=\alpha_2-\alpha_3,$
        $z=0,$
        $t=-(\alpha_2-\alpha_3) \beta_2,$
        and
        $u=-\alpha_3$
        we have the representative
            $\langle \nabla_2+\nabla_4,  \nabla_1\rangle.$     
      
 \item $ \beta_1=1, \alpha_1=0, \alpha_2= \alpha_3, 2\beta_2\neq  \beta_3:$
 by choosing 
        $x=1,$  
        $y=\frac{2}{2 \beta_2-\beta_3},$
        $z=0,$
        $t=-\frac{\beta_3}{2 \beta_2-\beta_3},$
        and
        $u=-\alpha_2$
        we have the representative
            $\langle \nabla_4,  \nabla_1+\nabla_2\rangle.$     
      
  \item $ \beta_1=1, \alpha_1=0, \alpha_2= \alpha_3, 2\beta_2=  \beta_3:$
 by choosing 
        $x=1,$  
        $y=1,$
        $z=0,$
        $t=-  \beta_2,$
        and
        $u=-\alpha_2$
        we have the representative
            $\langle \nabla_4,  \nabla_1 \rangle.$

 \item $ \beta_1=0, \alpha_1\neq 0,\beta_3\neq 0:$
 by choosing 
        $x=1,$  
        $y=\sqrt{\alpha_1},$
        $z=0,$
        $t=\frac{\alpha_2-\alpha_3}{\sqrt{\alpha_1}},$
        and
        $u=-2 \alpha_2+\alpha_3$
        we have the representative
            $\langle \nabla_1+\nabla_4, \alpha\nabla_2+ \nabla_3 \rangle.$     
 
  \item $ \beta_1=0, \alpha_1\neq 0,\beta_3= 0:$
 by choosing 
        $x=1,$  
        $y=\sqrt{\alpha_1},$
        $z=0,$
        $t=\frac{\alpha_2-\alpha_3}{\sqrt{\alpha_1}},$
        and
        $u=-2 \alpha_2+\alpha_3$
        we have the representative
            $\langle \nabla_1+\nabla_4, \nabla_2 \rangle.$

            \end{enumerate}
\item  $\alpha_4= 0, \beta_4=0, \alpha_1= 1, \beta_1=0.$ Then
\begin{enumerate} 
\item $ 2\beta_2\neq \beta_3, \beta_3\neq 0:$
 by choosing 
        $x=1,$  
        $y=1,$
        $z=0,$
        $t=\frac{\alpha_2 \beta_3-\alpha_3 \beta_2}{2 \beta_2-\beta_3},$
        and
        $u=0$
        we have the representative
            $\langle \nabla_1, \alpha \nabla_2+\nabla_3\rangle_{\alpha\neq \frac{1}{2}}.$  
            
\item $ 2\beta_2\neq \beta_3, \beta_3= 0:$
 by choosing 
        $x=1,$  
        $y=1,$
        $z=0,$
         $t=-\frac{\alpha_3}{2},$ 
        and
        $u=0$
        we have the representative
            $\langle \nabla_1,  \nabla_2\rangle.$  
            
 \item $ 2\beta_2= \beta_3, 2\alpha_2\neq \alpha_3 :$
 by choosing 
        $x=1,$  
        $y=\frac{1}{\alpha_3-2 \alpha_2},$
        $z=0,$
        $t=\frac{\alpha_2}{2 \alpha_2-\alpha_3},$
        and
        $u=0$
        we have the representative
            $\langle \nabla_1+\nabla_3,  \nabla_2+2\nabla_3 \rangle.$  
  \item $ 2\beta_2= \beta_3, 2\alpha_2= \alpha_3 :$   
      we have the representative
            $\langle \nabla_1,  \frac{1}{2}\nabla_2+\nabla_3 \rangle.$  
                           
            \end{enumerate}
  \end{enumerate}

Summarizing, we have the following distinct orbits:
\[ \langle \nabla_1,  \nabla_2\rangle, \ 
\langle \nabla_1, \alpha \nabla_2+\nabla_3\rangle,  \
\langle  \nabla_1,\nabla_2+\nabla_4\rangle, \ 
\langle \nabla_1, \nabla_4 \rangle, \ 
\langle   \nabla_1+\nabla_2, \nabla_4\rangle, \] 
\[\langle \nabla_1+ \nabla_2, \nabla_2+\nabla_4\rangle, \
\langle \nabla_1+\nabla_3,  \nabla_2+2\nabla_3 \rangle, \  
\langle \nabla_1+\nabla_4, \nabla_2 \rangle, \  
\langle \nabla_1+\nabla_4, \alpha \nabla_2+ \nabla_3 \rangle. \]

\bigskip

\subsection{$3$-dimensional central extensions of $\mathfrak{N}_1^{\mathbb{C}}$}
Consider  elements 
\[ \theta_1= \sum\limits_{i=1}^4 \alpha_i \nabla_i, \ 
 \theta_2= \sum\limits_{i=1}^4 \beta_i \nabla_i, \mbox{ and }
 \theta_3= \sum\limits_{i=1}^4 \gamma_i \nabla_i.\] 
We are interested in elements with 
$(\alpha_2, \alpha_3, \alpha_4,\beta_2, \beta_3, \beta_4,\gamma_2, \gamma_3, \gamma_4) \neq (0,0,0,0,0,0,0,0,0)$ and $(\alpha_1, \beta_1, \gamma_1) \neq (0,0,0).$
We provide the orbit of every possible case:
\begin{enumerate}
\item  $\alpha_4=1, \beta_1=\gamma_1=0, \alpha_1\neq 0, \beta_4=0, \gamma_4=0.$ Then
 by choosing 
 $x=\frac{1}{\sqrt[3]{\alpha_1}},$
 $y=1,$
    $z=0,$
        $t=0,$
        and
        $u=0$
        we have the representative
            $\langle \nabla_1+\nabla_4,  \nabla_2, \nabla_3 \rangle.$  

\item  $\alpha_4=1, \beta_1=1, \gamma_1=0, \alpha_1= 0, \beta_4=0, \gamma_4=0, 2 \gamma_2 \neq \gamma_3, 
\gamma_2 \neq \gamma_3, \gamma_3\neq 0.$ Then
 by choosing 
 $x=1,$
 $y=1,$
    $z=0,$
        $t=\frac{\beta_2 \gamma_3-\beta_3 \gamma_2}{2 \gamma_2-\gamma_3},$
        and
        $u=\frac{\alpha_2 \gamma_3-\alpha_3 \gamma_2}{\gamma_2-\gamma_3}$
        we have the representative
            $\langle \nabla_4, \nabla_1,  \alpha \nabla_2+ \nabla_3 \rangle_{\alpha\neq \frac{1}{2},1}.$  

\item  $\alpha_4=1, \beta_1=1, \gamma_1=0, \alpha_1= 0, \beta_4=0, \gamma_4=0,  
\gamma_2 \neq 0, \gamma_3= 0.$ Then
 by choosing 
 $x=1,$
 $y=1,$
    $z=0,$
        $t=-\frac{\beta_3 }{2},$
        and
        $u=-\alpha_3 $
        we have the representative
            $\langle \nabla_4, \nabla_1,   \nabla_2  \rangle.$

\item  $\alpha_4=1, \beta_1=1, \gamma_1=0, \alpha_1= 0, \beta_4=0, \gamma_4=0, 
\gamma_2 = \gamma_3, \gamma_3\neq 0, \alpha_2\neq \alpha_3.$ Then
 by choosing 
 $x=\alpha_3-\alpha_2,$
 $y=\alpha_2-\alpha_3,$
    $z=0,$
        $t=(\alpha_2-\alpha_3) (\beta_2-\beta_3),$
        and
        $u=0$
        we have the representative
            $\langle \nabla_3+\nabla_4, \nabla_1,   \nabla_2+ \nabla_3 \rangle.$

\item  $\alpha_4=1, \beta_1=1, \gamma_1=0, \alpha_1= 0, \beta_4=0, \gamma_4=0, 
\gamma_2 = \gamma_3, \gamma_3\neq 0, \alpha_2= \alpha_3.$ Then
 by choosing 
 $x=1,$
 $y=1,$
    $z=0,$
        $t=\beta_2-\beta_3,$
        and
        $u=0$
        we have the representative
            $\langle \nabla_4, \nabla_1,   \nabla_2+ \nabla_3 \rangle.$              
 
\item  $\alpha_4=1, \beta_1=1, \gamma_1=0, \alpha_1= 0, \beta_4=0, \gamma_4=0, 
2 \gamma_2 = \gamma_3, \gamma_3\neq 0,  2 \beta_2 \neq \beta_3, 2 \alpha_2 \neq \alpha_3.$ Then
 by choosing 
 $x=(\alpha_3-2 \alpha_2) (\beta_3-2 \beta_2),$
 $y=(\alpha_3-2 \alpha_2)^2 (\beta_3-2 \beta_2),$
    $z=0,$
        $t=0,$
        and
        $u=0$
        we have the representative
            $\langle \nabla_3+ \nabla_4, \nabla_1+ \nabla_3,   \nabla_2+ 2\nabla_3 \rangle.$

\item  $\alpha_4=1, \beta_1=1, \gamma_1=0, \alpha_1= 0, \beta_4=0, \gamma_4=0, 
2 \gamma_2 = \gamma_3, \gamma_3\neq 0,  2 \beta_2 \neq \beta_3, 2 \alpha_2 = \alpha_3.$ Then
 by choosing 
 $x=2\beta_2- \beta_3,$
 $y=\beta_3-2 \beta_2,$
    $z=0,$
        $t=0,$
        and
        $u=0$
        we have the representative
            $\langle  \nabla_4, \nabla_1+ \nabla_3,   \nabla_2+ 2\nabla_3 \rangle.$

\item  $\alpha_4=1, \beta_1=1, \gamma_1=0, \alpha_1= 0, \beta_4=0, \gamma_4=0, 
2 \gamma_2 = \gamma_3, \gamma_3\neq 0,  2 \beta_2 = \beta_3.$ Then
 by choosing 
 $x=1,$
 $y=1,$
    $z=0,$
        $t=0$
        and
        $u=\alpha_3-2 \alpha_2$
        we have the representative
            $\langle  \nabla_4, \nabla_1,   \nabla_2+ 2\nabla_3 \rangle.$

\item  $\alpha_4=0, \beta_4=0, \gamma_4=0.$ Then
        we have the representative
            $\langle  \nabla_1, \nabla_2,   \nabla_3 \rangle.$              
                        
 \end{enumerate}

Note that

\begin{longtable}{lcccl}
$\left\langle  \begin{array}{l}
\nabla_1+ \nabla_3,  \\ 
\nabla_2+ 2\nabla_3,\\
\nabla_3+ \nabla_4
\end{array}\right\rangle$ &
$\to$ &
$\left\{ 
\begin{array}{lll}
x=1 & y=1 & z=0\\
t=0 & u=1 &
\end{array}
\right\}$

$\to$ &
$\left\langle  \begin{array}{l}
\nabla_1+ \nabla_3,  \\ 
\nabla_2+ 2\nabla_3,\\ \nabla_4 
\end{array} \right\rangle.$
\end{longtable}

Summarizing, we have the following distinct orbits:
\[ \langle  \nabla_1, \nabla_2,   \nabla_3 \rangle, \              
\langle \nabla_1,   \nabla_2, \nabla_4  \rangle, \ 
\langle \nabla_1,  \alpha \nabla_2+ \nabla_3, \nabla_4 \rangle, \
\langle  \nabla_1,   \nabla_2+ \nabla_3, \nabla_3+\nabla_4 \rangle,  \]  
\[ 
\langle  \nabla_1+ \nabla_3,   \nabla_2+ 2\nabla_3, \nabla_4 \rangle, \  
\langle \nabla_1+\nabla_4,  \nabla_2, \nabla_3 \rangle.\]

\subsection{$4$-dimensional central extensions of $\mathfrak{N}_1^{\mathbb{C}}$}
 There is only one $4$-dimensional central extension defined by 
 \[\langle \nabla_1, \nabla_2, \nabla_3, \nabla_4 \rangle.\]

\subsection{Classification theorem}
Summarizing all results regarding to classification of distinct orbits, 
we have the classification of all central extensions of the algebra $\mathfrak{N}_1^{\mathbb{C}}.$
Note that, we are interested only in non-trivial central extensions,
which are non-split and can not be considered as central extensions of an algebra of smaller dimension than $\mathfrak{N}_1^{\mathbb{C}}.$

\begin{theorem}
Let $[\mathfrak{N}_1^{\mathbb{C}}]^i$ be an $i$-dimensional non-trivial central extension of the Zinbiel algebra $\mathfrak{N}_1^{\mathbb{C}}.$ Then $[\mathfrak{N}_1^{\mathbb{C}}]^i$ is isomorphic to one algebra from the following list:

\begin{enumerate}
    \item if $i=1:$
    \begin{longtable}{lllllll}
$[\mathfrak{N}_1^{\mathbb{C}}]^1_{01}$ &$:$& $ e_1^2=e_2$ & $e_1 e_2 =e_4$ &$e_2 e_1 =2e_4$ &$e_3 e_3 =e_4$\\

$[\mathfrak{N}_1^{\mathbb{C}}]^1_{02}$ &$:$& $ e_1^2=e_2$ & $e_1 e_2 =e_4$ &$e_1 e_3 =e_4$ &$e_2 e_1 =2e_4$ \\

    \end{longtable}

    \item if $i=2:$
    \begin{longtable}{llllllllll}

$[\mathfrak{N}_1^{\mathbb{C}}]^2_{01}$ &$:$& $ e_1^2=e_2$ &$e_1 e_2 =e_4$ &$e_1 e_3 =e_5$&$e_2 e_1 =2e_4$ \\

$[\mathfrak{N}_1^{\mathbb{C}}]^2_{02}(\alpha)$ &$:$& $ e_1^2=e_2$ &$e_1 e_2 =e_4$ &$e_1 e_3 =\alpha e_5$&$e_2 e_1 =2e_4$&$e_3 e_1 =e_5$\\

$[\mathfrak{N}_1^{\mathbb{C}}]^2_{03}$ &$:$& $ e_1^2=e_2$ &$e_1 e_2 =e_4$ &$e_1 e_3 =e_5$&$e_2 e_1 =2e_4$&$e_3 e_3 =e_5$\\

$[\mathfrak{N}_1^{\mathbb{C}}]^2_{04}$ &$:$& $ e_1^2=e_2$&$e_1 e_2 =e_4$&$e_2 e_1 =2e_4$ &$e_3 e_3 =e_5$\\

$[\mathfrak{N}_1^{\mathbb{C}}]^2_{05}$ &$:$& $ e_1^2=e_2$ &$e_1 e_2 =e_4$ &$e_1 e_3 =e_4$&$e_2 e_1 =2e_4$ &$e_3 e_3 =e_5$ \\

$[\mathfrak{N}_1^{\mathbb{C}}]^2_{06}$ &$:$& $ e_1^2=e_2$ &$e_1 e_2 =e_4$ &$e_1 e_3 =e_4+e_5$ &$e_2 e_1 =2e_4$ &$e_3 e_3 =e_5$\\

$[\mathfrak{N}_1^{\mathbb{C}}]^2_{07}$ &$:$& $ e_1^2=e_2$ &$e_1 e_2 =e_4$ &$e_1 e_3 =e_5$ &$e_2 e_1 =2e_4$ &\multicolumn{2}{l}{$e_3 e_1 =e_4+2e_5$}\\

$[\mathfrak{N}_1^{\mathbb{C}}]^2_{08}$ &$:$& $ e_1^2=e_2$ &$e_1 e_2 =e_4$ &$e_1 e_3 =e_5$ &$e_2 e_1 =2e_4$ &$e_3 e_3 =e_4$\\

$[\mathfrak{N}_1^{\mathbb{C}}]^2_{09}(\alpha)$ &$:$& $ e_1^2=e_2$ &$e_1 e_2 =e_4$
&$e_1 e_3 =\alpha e_5$&$e_2 e_1 =2e_4$ &$e_3 e_1 =e_5$ &$e_3 e_3 =e_4$

    \end{longtable}

    \item if $i=3:$
    \begin{longtable}{llllllllll}

$[\mathfrak{N}_1^{\mathbb{C}}]^3_{01}$ &$:$& $ e_1^2=e_2$ &$e_1 e_2 =e_4$&$e_1 e_3 =e_5$ &$e_2 e_1 =2e_4$&$e_3 e_1 =e_6$\\

$[\mathfrak{N}_1^{\mathbb{C}}]^3_{02}$ &$:$& $ e_1^2=e_2$ &$e_1 e_2 =e_4$&$e_1 e_3 =e_5$ &$e_2 e_1 =2e_4$&$e_3 e_3 =e_6$\\

$[\mathfrak{N}_1^{\mathbb{C}}]^3_{03}(\alpha)$ &$:$& $ e_1^2=e_2$ &$e_1 e_2 =e_4$ &$e_1 e_3 =\alpha e_5$ &$e_2 e_1 =2e_4$ &$e_3 e_1 =e_5$ &$e_3 e_3 =e_6$ \\

$[\mathfrak{N}_1^{\mathbb{C}}]^3_{04}$ &$:$& $ e_1^2=e_2$ &$e_1 e_2 =e_4$&$e_1 e_3 =e_5$ &$e_2 e_1 =2e_4$ &$e_3 e_1 =e_5+e_6$ &$e_3 e_3 =e_6$ \\

$[\mathfrak{N}_1^{\mathbb{C}}]^3_{05}$ &$:$& $ e_1^2=e_2$ &$e_1 e_2 =e_4$&$e_1 e_3 =e_5$ &$e_2 e_1 =2e_4$&$e_3 e_1 =e_4+2e_5$ &$e_3 e_3 =e_6$\\

$[\mathfrak{N}_1^{\mathbb{C}}]^3_{06}$ &$:$& $ e_1^2=e_2$ &$e_1 e_2 =e_4$&$e_1 e_3 =e_5$ &$e_2 e_1 =2e_4$&$e_3 e_1 =e_6$ &$e_3 e_3 =e_4$\\

    \end{longtable}

    \item if $i=4:$
    \begin{longtable}{lllllllllll}

$[\mathfrak{N}_1^{\mathbb{C}}]^4_{01}$ &$:$& $ e_1^2=e_2$ &$e_1 e_2 =e_4$&$e_1 e_3 =e_5$  &$e_2 e_1 =2e_4$ &$e_3 e_1 =e_6$ &$e_3 e_3 =e_7$\\

    \end{longtable}

\end{enumerate}
\end{theorem}



\section{Central extensions of $\mathfrak{N}_1$} 
The multiplication table of the  $3$-dimensional Zinbiel algebra $\mathfrak{N}_1$ is given  by:
\[e_1e_2=  e_3, \ e_2e_1= -e_3 .\]
 The automorphism group of $\mathfrak{N}_1$ consists of invertible matrices of the form
\[\phi=\begin{pmatrix} 
x & y & 0 \\
z & w & 0 \\
t & p & xw-yz 
\end{pmatrix}. \]
The cohomology space of $\mathfrak{N}_1$ is given  by: 
\[ \nabla_1= [\Delta_{11}], \nabla_2=[\Delta_{12}], \nabla_3= [\Delta_{22}], \nabla_4= [\Delta_{13}], \nabla_5= [\Delta_{23}].\]
\medskip
Since
\[\begin{pmatrix}
x & y & 0\\
z & w & 0\\
t & p & xw-yz
\end{pmatrix}^t\begin{pmatrix}
\alpha_1 & \alpha_2 & \alpha_4\\
0 & \alpha_3 & \alpha_5\\
0 & 0 & 0
\end{pmatrix}\begin{pmatrix}
x & y & 0\\
z & w & 0\\
t & p & xw-yz
\end{pmatrix}=
\begin{pmatrix}
\alpha_1^* & \alpha_2^*-\beta & \alpha_4^*\\
\beta & \alpha_3^* & \alpha_5^*\\
0 & 0 & 0
\end{pmatrix},\]
where

\begin{longtable}{lcl}
$\alpha_1^*$&$=$&$\alpha_1x^2+(\alpha_2x+\alpha_3z)z+(\alpha_4x+\alpha_5z)t$\\
$\alpha_2^*$&$=$&$\alpha_1xy+(\alpha_2x+\alpha_3z)w+(\alpha_4x+\alpha_5z)p+\alpha_1xy+(\alpha_2y+\alpha_3w)z+(\alpha_4y+\alpha_5w)t$\\
$\alpha_3^*$&$=$&$\alpha_1y^2+(\alpha_2y+\alpha_3w)w+(\alpha_4y+\alpha_5w)p$\\
$\alpha_4^*$&$=$&$(\alpha_4x+\alpha_5z)(wx-yz)$\\
$\alpha_5^*$&$=$&$(\alpha_4y+\alpha_5w)(wx-yz).$
\end{longtable}
we obtain that the action of $\operatorname{Aut}\left(\mathfrak{N}_1\right)$ on a subspace $ \langle \sum\limits_{i=1}^5 \alpha_i \nabla_i  \rangle$  is given by
$ \langle \sum\limits_{i=1}^5 \alpha_i^* \nabla_i  \rangle.$

\subsection{$1$-dimensional central extensions of $\mathfrak{N}_1$}
We are only interested in cocycles with $(\alpha_4,\alpha_5) \neq (0,0).$


\begin{lemma}
    The 1-dimensional subspaces $\langle \nabla_4 \rangle$ and $\langle \nabla_3+\nabla_4\rangle$ generate all pairwise distinct orbits with $(\alpha_4, \alpha_5)\neq (0,0)$.
\end{lemma}

\begin{proof}
    It is easy to see, that we may assume $\alpha_4\neq 0.$
    \begin{enumerate}
        \item By aplying an automorphism $\phi$ in $\langle \nabla_4 \rangle$ with $x=1$, $y=\frac {\alpha_5}{\alpha_{4}}$, $z=1$, $w=\frac {\alpha_4^2+\alpha_5}{\alpha_4}$, $t=\alpha_1$, $p=\frac{ \alpha_2\alpha_4-\alpha_1\alpha_5}{\alpha_4}$, we obtain all subspaces $ \langle \sum\limits_{i=1}^5 \alpha_i \nabla_i  \rangle$ with $\alpha_4^2\alpha_3+ \alpha_5 (\alpha_1 \alpha_5-\alpha_2 \alpha_4)=0$.
        
        \item By aplying an automorphism $\phi$ in $\langle \nabla_3+\nabla_4 \rangle$ with $y=\frac{x\alpha_5}{\alpha_4}$, $z=1$, $w=\frac{x^2\alpha_5+\alpha_4}{x^2 \alpha_4}$, $t=\frac{\alpha_1-1}{x}$, $p=\frac{x^2(\alpha_1\alpha_5-\alpha_2\alpha_4+\alpha_5) 2+\alpha_4^2}{x^3\alpha_4}$,  $p=-\frac{x^2(\alpha_1\alpha_5-\alpha_2\alpha_4+\alpha_5) +2\alpha_4^2}{x^3\alpha_4}$,  where $x$ is a fourth root of $\frac{\alpha_4^4}{\alpha_3\alpha_4^2+\alpha_5(\alpha_1\alpha_5-\alpha_2\alpha_4)}$, we obtain all subspaces $ \langle \sum\limits_{i=1}^5 \alpha_i \nabla_i  \rangle$ with $\alpha_4^2\alpha_3+ \alpha_5 (\alpha_1 \alpha_5-\alpha_2 \alpha_4)\neq 0$.
    \end{enumerate}
    All cases have been considered and the lemma is proved.
\end{proof}

\subsection{$2$-dimensional central extensions of $\mathfrak{N}_1$}


  From the previous section, all orbits are generated by two-dimensional subspaces of the form $\langle \theta_1, \theta_2 \rangle$, where $\theta_1= \sum\limits_{i=1}^5 \alpha_i \nabla_i$ and $\theta_2=\nabla_4$ or $\theta_2=\nabla_3+\nabla_4$.
	Of course, we may assume $\alpha_4=0$.

\begin{lemma}
	The following subspaces generate all pairwise distinct orbits.
	
	\begin{multicols}{2}
		\begin{enumerate}[$\mathcal{O}_1=$]
			\item $\langle \nabla_1, \nabla_4\rangle$
			
			\item $\langle \nabla_2, \nabla_4\rangle$
			
			\item $\langle \nabla_3, \nabla_4\rangle$
			
			\item $\langle \nabla_1+\nabla_3, \nabla_4\rangle$
			
			\item $\langle \nabla_5, \nabla_4\rangle$
			
			\item $\langle \nabla_1+\nabla_5, \nabla_4\rangle$
			
			\item $\langle \nabla_2+\nabla_5, \nabla_4\rangle$
			
			\item ${\langle \nabla_1+\nabla_3+\nabla_5, \nabla_4\rangle}$
			
			\item $\langle \nabla_1, \nabla_3+\nabla_4\rangle$
			
			\item [$\mathcal{O}_{10}=$]$\langle \nabla_2, \nabla_3+\nabla_4\rangle$

		\end{enumerate}
		
	\end{multicols}
\end{lemma}

\begin{proof}
	First, we observe that any orbit is generated by a subspace contained in one of the sets below:
	
	$S_1=\{\langle \alpha_1\nabla_1+\alpha_2\nabla_2+\alpha_3\nabla_3+\alpha_5\nabla_5, \nabla_4\rangle \,|\, \alpha_i\in \mathbb{C} \text{ are not all zero }\}$
	
	$S_2=\{\langle \alpha_1\nabla_1+\alpha_2\nabla_2+\alpha_3\nabla_3+\alpha_5\nabla_5, \nabla_3+\nabla_4\rangle \,|\, \alpha_i\in \mathbb{C} \text{ are not all zero }\}$
	
	In order to prove the result, we need to show that all elements in $S_1\cup S_2$ lie in one of the orbits $\mathcal O_1,\dots, \mathcal O_{10}$, and that these orbits are pairwise distinct.

\begin{enumerate}
	
	\item\label{1} Applying the automorphism $\phi$ in the subspace $\mathcal{O}_1= \langle \nabla_1, \nabla_4\rangle$ with $y=t=p=0$, we obtain any subspace in $S_1$ with $\alpha_2=\alpha_3=\alpha_5=0$. 

	\item\label{2} Applying the automorphism $\phi$ in the subspace $\mathcal{O}_2=\langle \nabla_2, \nabla_4 \rangle$ with $y=p=t=0$, $z=\alpha_1$, $w=\alpha_2$, we obtain any subspace in $S_1$ with $\alpha_2\neq0$, $\alpha_3=\alpha_5=0$.	Moreover, any automorphism $\phi$ applied to this orbit does not contain $\mathcal O_1$. Indeed, by applying an automorphism $\phi$ in $\langle \nabla_2, \nabla_4 \rangle$, to obtain an element of $S_1$ we must have $y=0$, $p=0$ and $t=0$, and this yields an element of the form $\langle z\nabla_1+w\nabla_2, \nabla_4 \rangle$, where $w$ must be non-zero. In particular, it does not contain the subspace $\mathcal{O}_1$. 
		
	\item Applying the automorphism $\phi$ in the subspace $\mathcal{O}_3=\langle \nabla_3, \nabla_4 \rangle$ with $x=w=1$, $y=p=t=0$,  $z=\frac{\alpha_2}{2}$, we obtain any subspace in $S_1$ with $\alpha_5=0$, $\alpha_3\neq 0$ and $\alpha_2^2=4\alpha_1\alpha_3$. Arguing as in (\ref{2}), any automorphism applied to this orbit does not contain the subspaces $\mathcal{O}_1$ and $\mathcal{O}_2$.
	
	\item Applying the automorphism $\phi$ in the subspace $\mathcal{O}_4=\langle \nabla_1+\nabla_3, \nabla_4 \rangle$ with  $x=\frac{\sqrt{4\alpha_1\alpha_3-\alpha_2^2}}{2\sqrt{\alpha_3}}$,  $y=t=p=0$, $z=\frac{\alpha_2}{2\sqrt{\alpha_3}}$ and $w=\sqrt{\alpha_3}$, we obtain any subspace in $S_1$ with $\alpha_5=0$, $\alpha_3\neq0$ and $\alpha_2^2\neq4\alpha_1\alpha_3$. As above, one verifies that this orbit does not contain the subspaces $\mathcal{O}_1$, $\mathcal{O}_2$ and $\mathcal{O}_3$.
	

	\item Applying the automorphism $\phi$ in the subspace $\mathcal{O}_5=\langle \nabla_5, \nabla_4 \rangle$ with  $y=t=p=0$, we obtain any subspace of $S_1$ with $\alpha_5=1$ and $\alpha_1=\alpha_2=\alpha_3=0$. Moreover, by applying an automorphism $\phi$ with $w=0$, one obtains an element of $S_1$ if and only if $p=t=0$, and the resulting subspace will be $\langle \nabla_5, \nabla_4 \rangle$. On the other hand, by applying an arbitrary automorphism $\phi$, with $w\neq0$,$y=0$, we obtain the subspace $\langle \frac{t}{\delta}\nabla_2+\frac{p}{\delta}\nabla_3+\nabla_5, \frac{t}{\delta}\nabla_1+\frac{p}{\delta}\nabla_2+\nabla_4 \rangle$, where $\delta=xw-yz$. This is an element of $S_1$ if and only if $p=t=0$, and this also yields the subspace $\langle \nabla_5, \nabla_4 \rangle$. As a consequence, this orbit contains no other subspaces of $S_1$ other than $\mathcal O_5$.
	
	\item Applying the automorphism $\phi$ in the subspace $\mathcal{O}_6=\langle \nabla_1+\nabla_5, \nabla_4 \rangle$ with $x=\alpha_1$, $y=t=p=0$, $w=1$, we obtain any subspace of $S_1$ with $\alpha_5=1$, $\alpha_1\neq0$, $\alpha_2=\alpha_3=0$. Arguing similarly as above, we conclude that this orbit does not contain the subspaces $\mathcal O_1, \dots, \mathcal O_5$.
	

	\item\label{7} Applying the automorphism $\phi$ in the subspace $\mathcal{O}_7=\langle \nabla_2+\nabla_5, \nabla_4 \rangle$ with $x=1$, $z=\frac{\alpha_1}{\alpha_2^2}$, $y=t=p=0$, $w=\frac{1}{\alpha_2}$, we obtain any subspace of $S_1$ with $\alpha_5=1$, $\alpha_3=0$, and $\alpha_2\neq0$. \\
		Applying the automorphism $\phi$ in the same subspace with 
		$x=\frac{\alpha_2}{2}$, $y=\alpha_3$, $z=\frac{1}{\alpha_3}$, $w=0$, $p=-\alpha_3$,  $t=-\frac{\alpha_2}{2}$, we obtain any subspace of $S_1$ with $\alpha_5=1$, $\alpha_3\neq0$ and $4\alpha_1\alpha_3=\alpha_2^2$.\\		
	By applying an arbitrary automorphism $\phi$ in this orbit, it results in a subspace lying in  $S_1$ if and only if $y=0$, or $w=0$, which result in the cases already considered in (\ref{7}).
		 	
	\item\label{8} Applying the automorphism $\phi$ in the subspace $\mathcal{O}_8=\langle \nabla_1+\nabla_3+\nabla_5, \nabla_4 \rangle$ with $x=\frac{1}{\alpha_3}$, $y=p=t=0$, $z=\frac{\alpha_2}{\alpha_3\sqrt{4\alpha_1\alpha_3-\alpha_2^2}}$, $w=\frac{2}{\sqrt{4\alpha_1\alpha_3-\alpha_2^2}}$, we obtain any subspace in $S_1$ with $\alpha_5=1$, $\alpha_3\neq0$ and $4\alpha_1\alpha_3\neq \alpha_2^2$.\\
	By applying an arbitrary automorphism $\phi$ in $\mathcal O_8$, it results in a subspace lying in $S_1$ if and only if $y=t=p=0$. And the resulting subspaces are different from the subspaces $\mathcal O_1, \dots, \mathcal O_7$.

\end{enumerate}

As a consequence of the above, we obtain that any subspace in $S_1$ lies in the orbit generated by one (and only one) of the following subspaces:

	\begin{multicols}{2}
	\begin{enumerate}[$\mathcal{O}_1=$]
		\item $\langle \nabla_1, \nabla_4\rangle$
		
		\item $\langle \nabla_2, \nabla_4\rangle$
		
		\item $\langle \nabla_3, \nabla_4\rangle$
		
		\item $\langle \nabla_1+\nabla_3, \nabla_4\rangle$
		
		\item $\langle \nabla_5, \nabla_4\rangle$
		
		\item $\langle \nabla_1+\nabla_5, \nabla_4\rangle$
		
		\item $\langle \nabla_2+\nabla_5, \nabla_4\rangle$
		
		\item ${\langle \nabla_1+\nabla_3+\nabla_5, \nabla_4\rangle}$
		
	\end{enumerate}
	
\end{multicols}
 
Moreover, as mentioned in (\ref{1})-(\ref{8}), $\mathcal O_1, \dots, \mathcal O_8$ generate pairwise distinct orbits.

\begin{enumerate}[(1)]
	  \setcounter{enumi}{8}
	\item Applying the automorphism $\phi$ in the subspace $\mathcal O_9=\langle \nabla_1,\nabla_3+\nabla_4 \rangle$, with $x=w=1$, $y=z=p=t=0$,  we obtain any subspace in $S_2$ with $\alpha_2=\alpha_3=\alpha_5=0$.	Moreover, any automorphism $\phi$ applied to this orbit does not contain elements of $S_1$. In particular, the orbit generated by $\mathcal O_9$ is different from the ones generated by $\mathcal O_1, \dots, \mathcal O_8$.

	\item Applying the automorphism $\phi$ in the subspace $\mathcal O_{10}=\langle \nabla_2,\nabla_3+\nabla_4 \rangle$, with $x=1$, $y=0$, $z=\alpha_1$, $w=1$, $p=-2\alpha_1$, $t=-{\alpha_1^2}$, we obtain any subspace in $S_2$ with $\alpha_5=\alpha_3=0$, $\alpha_2=1$. Arguing as in the previous case, we obtain that this orbit does not contain $\mathcal O_1, \dots, \mathcal O_8$. Moreover, if $\phi$ is an automorphism, applying it to $\mathcal O_{10}$ it results in an element of $S_2$ if and only if $y=0$, $w=x^2$, $p=-2xz$ and $t=-\frac{z^2}{x}$. As a consequence, the resulting subspace cannot be $\mathcal O_{9}$.
	
	\item Applying the automorphism $\phi$ in the subspace $\mathcal O_3=\langle \nabla_3,\nabla_4 \rangle$, with $x=1$, $y=0$, $z=\frac{\alpha_2}{2}$, $w=1$, $p=-{\alpha_2}$, $t=-\frac{\alpha_2^2}{4}$, we obtain any subspace in $S_2$ with $\alpha_5=0$, $\alpha_3=1$ and $4\alpha_1-\alpha_2^2=0$.
	
	\item Applying the automorphism $\phi$ in the subspace $\mathcal O_4=\langle \nabla_1+\nabla_3,\nabla_4 \rangle$, with $x=1$, $y=0$, $z=\frac{\alpha_2}{\sqrt{4\alpha_1-\alpha_2^2}}$, $w=\frac{2}{\sqrt{4\alpha_1-\alpha_2^2}}$, $p=-\frac{2\alpha_2}{\sqrt{4\alpha_1-\alpha_2^2}}$, $t=-\frac{2\alpha_1}{\sqrt{4\alpha_1-\alpha_2^2}}$, we obtain any subspace in $S_2$ with $\alpha_5=0$, $\alpha_3=1$ and $4\alpha_1-\alpha_2^2\neq 0$.
	
	
	\item Applying the automorphism $\phi$ in the subspace $\mathcal O_6=\langle \nabla_1+\nabla_5,\nabla_4 \rangle$, with $x=\frac{\alpha_3}{3}$, $y=-1$, $z=1$, $w=0$, $p=\frac{2\alpha_3}{3}$, $t=-\frac{\alpha_3^2}{9}$, we obtain any subspace in $S_2$ with $\alpha_5=1$, and $\alpha_1,\alpha_2,\alpha_3\in \mathbb C$ satisfying $3\alpha_2+\alpha_3^2=0$  and $27\alpha_1-\alpha_3^3=0$.

	\item Applying the automorphism $\phi$ in the subspace $\mathcal O_8=\langle \nabla_1+\nabla_3+\nabla_5,\nabla_4 \rangle$, with $x=\frac{-\alpha_3^3-\alpha_3s+27\alpha_1}{s^2}$, $y=\frac{3}{s}$, $z=-\frac{\sqrt{3}(-\alpha_3^3+\alpha_3s+27\alpha_1)}{s^2}$, $w=\frac{3\sqrt{3}}{s}$, $p=\frac{2\sqrt{3}(-2\alpha_3^4-s^2+54\alpha_1\alpha_3)}{s^3}$ and $t=\frac{2}{\sqrt{3}}\frac{\alpha_3^5-s\alpha_3^3+\alpha_3s^2-27\alpha_1\alpha_3^2+27s\alpha_1}{s^3}$, where $s$ satisfies $s^3+(27\alpha_1-\alpha_3^3)^2=0$, we obtain any subspace in $S_2$ with $\alpha_5=1$ and $\alpha_1,\alpha_2, \alpha_3\in \mathbb C$ satisfying $3\alpha_2+\alpha_3^2=0$   and $27\alpha_1-\alpha_3^3\neq0$.	

	
	\item Last, we remark that the subspaces $\langle \nabla_2+\nabla_5,\nabla_3+\nabla_4 \rangle$ and $\mathcal O_8=\langle \nabla_1+\nabla_3+\nabla_5,\nabla_4 \rangle$ generate the same orbit. Indeed, by applying the automorphism $\phi$ in the former with $x=w=t=0$, $z=1$ and $y=i$, $p=-i$, we obtain the later.

	Applying the automorphism $\phi$ in the subspace $\langle \nabla_2+\nabla_5,\nabla_3+\nabla_4 \rangle$, with $x=\frac{2-3r^2}{\sqrt{4-27\alpha^2}}$, $y=\frac{-3r}{\sqrt{4-27\alpha^2}}$, $z=\frac{2r+6\alpha-9r^2\alpha}{4-27\alpha^2}$, $w=\frac{9\alpha r-6r^2+4}{4-27\alpha^2}$, $p=\frac{r}{\sqrt{4-27\alpha^2}}$ and $t=\frac{r^2}{\sqrt{4-27\alpha^2}}$, where $r$ satisfies $r^3-r-\alpha=0$, we obtain any element of the type $\langle\alpha \nabla_1+ \nabla_2+\nabla_5,\nabla_3+\nabla_4 \rangle$ in $S_2$ with $\alpha\in \mathbb{C}\setminus \{-\frac{2}{3\sqrt{3}}, \frac{2}{3\sqrt{3}}\}$.
	
	The element $\langle\frac{2}{3\sqrt{3}} \nabla_1+ \nabla_2+\nabla_5,\nabla_3+\nabla_4 \rangle$ generate the same orbit as $\mathcal O_7=\langle \nabla_2+\nabla_5,\nabla_4 \rangle$. Indeed, by applying the automorphism $\phi$ in the later with $x = {-3}$, $y = -{3\sqrt{3}}$, $z = \frac{2}{3\sqrt{3}}$, $w = -\frac{1}{3}$, $p= \sqrt{3}$ and $ t = 2$, we obtain the former.
	
	The element $\langle-\frac{2}{3\sqrt{3}} \nabla_1+ \nabla_2+\nabla_5,\nabla_3+\nabla_4 \rangle$ generate the same orbit as $\mathcal O_7= \langle \nabla_2+\nabla_5,\nabla_4 \rangle$. Indeed, by applying the automorphism $\phi$ in the later with $x = {3}$, $y = -{3\sqrt{3}}$, $z = -\frac{2}{3\sqrt{3}}$, $w = -\frac{1}{3}$, $p= \sqrt{3}$ and $ t = -2$, we obtain the former.
	
	As a consequence, for any $\alpha\in \mathbb C$, the subspaces $\langle\alpha \nabla_1+ \nabla_2+\nabla_5,\nabla_3+\nabla_4 \rangle$ belog to some of the orbits considered above.

 	Let now $\alpha_1$ $\alpha_2$, $\alpha_3\in \mathbb{C}$ satisfy $\alpha_3^2+3\alpha_2\neq 0$ and define $\alpha=\frac{2\alpha_3^3+9\alpha_2\alpha_3+27\alpha_1}{3\sqrt{3}(\alpha_3^2+3\alpha_2)^{3/2}}$. By  applying the automorphism $\phi$ in $\langle\alpha \nabla_1+ \nabla_2+\nabla_5,\nabla_3+\nabla_4 \rangle$, with $x=\frac{\sqrt{3}}{\sqrt{\alpha_3^2+3\alpha_2}}$, $y=0$, $z=\frac{-\alpha_3}{\alpha_3^2+3\alpha_2}$, $w=\frac{3}{\alpha_3^2+3\alpha_2}$,  $p=\frac{2\sqrt{3}\alpha_3}{(\alpha_3^2+3\alpha_2)^{3/2}}$ and $t=\frac{-\alpha_3^2}{\sqrt{3}(\alpha_3^2+3\alpha_2)^{3/2}}$, we obtain any element in $S_2$ with $\alpha_3^2+3\alpha_2\neq 0$.

\end{enumerate}

	Hence, all subspaces lie in some of the orbits generated by the subspaces $\mathcal{O}_1, \dots, \mathcal{O}_{10}$, which generate pairwise distinct orbits. The Lemma is proved.

\end{proof}

\subsection{$3$-dimensional central extensions of $\mathfrak{N}_1$}

We may assume that a $3$-dimensional subspace is generated by
$\theta_1, \theta_2$ and $\theta_3$, where $\langle \theta_2, \theta_3\rangle\in\{ \mathcal O_1, \dots, \mathcal O_{10}\}$.

Before the main result of this subsection, we need the following technical result.

\begin{lemma}\label{solution}
	Let $\alpha_1, \alpha_2\in \mathbb{C}\setminus \{0\}$. The system of equations	
	\begin{equation}\label{system}\left\{\begin{array}{rcl}
		\frac{w(w^2-3y^2)}{(w^2+y^2)^2} & = & \alpha_1\\
		\frac{y(3w^2-y^2)}{(w^2+y^2)^2} & = & \alpha_2\\
		\end{array} \right.\end{equation}		
has a solution if and only if $\alpha_1^2+\alpha_2^2\neq0$.
\end{lemma}

\begin{proof}
	Let $(y,w)$ be a solution of (\ref{system}). Then $yw\neq0$. Let $k=\frac{y}{w}$. Substituting in (\ref{system}), we obtain 
	$\left\{\begin{array}{rcl}	
	{\frac {1-3k^2}{w \left( {k}^{2}+1 \right) ^{2}}}& = & \alpha_1\\
	{\frac {( 3-k^2 ) k}{w \left( {k}^{2}+1 \right) ^{2}}}& = & \alpha_2\\
	\end{array} \right.$\\	
Then, 
\begin{equation}\label{w}
w={\frac {1-3k^2}{\alpha_1 ( {k}^{2}+1 ) ^{2}}}={\frac {(3-k^2) k}{\alpha_2 ( {k}^{2}+1 ) ^{2}}}\end{equation} and $k$ must satisfy the equation:	
	$\alpha_1k^3-3\alpha_2k^2-3\alpha_1 k+\alpha_2=0$.
	If we denote $q=\frac{\alpha_2}{\alpha_1}$, the above equation becomes \begin{equation}\label{poly}
	k^3-3qk^2-3k+q=0.
	\end{equation}
	
	Once the above equation has a solution $k\not\in \{0,i,-i\}$, we obtain $w$ in (\ref{w}), and we find a solution to the system (\ref{system}). Of course, $k=0$ is not a solution to this polynomial equation, otherwise $\alpha_2=0$. 
	
	Observe that $k=i$ is a solution to (\ref{poly}) if and only if $q=i$, and $k=-i$ is a solution to (\ref{poly}) if and only if $q=-i$. Moreover, if $q=i$, (\ref{poly}) becomes $(k-i)^3=0$ and if $q=-i$, it  becomes $(k+i)^3=0$
	
	As a consequence, equation (\ref{poly}) has a solution different from $i, -i, 0$ if and only if $q\neq \pm i$, which is equivalent to $\alpha_1^2+\alpha_2^2\neq0$.	
\end{proof}

The proof of the next lemma is similar to the above, and thus is omitted.

\begin{lemma}\label{eq2}
	Let $\alpha_1, \alpha_2\in \mathbb{C}\setminus \{0\}$. The system of equations	
	\begin{equation}\label{system2}
	\left\{\begin{array}{rcl} {\frac {({w}^{2}+{y}^{2})^3 \alpha_{2}^{2}}{4 \left( {w}^{2
			}-3\,{y}^{2} \right) ^{2}{w}^{2}}} & = & \alpha_{1}\\
	\frac{(3w^2-y^2)y}{y^2+w^2} & = & 1
	\end{array}\right.	
	\end{equation}		
	has a solution if and only if $4\alpha_1-\alpha_2^2\neq0$.
\end{lemma}

\begin{lemma}
	The following subspaces generate all pairwise distinct orbits.
	
	\begin{multicols}{2}
		\begin{enumerate}[$\mathcal{P}_1=$]
			\item $\langle \nabla_2, \nabla_1, \nabla_4\rangle$
			
			\item $\langle \nabla_3, \nabla_1, \nabla_4\rangle$
			
			\item $\langle \nabla_5, \nabla_1, \nabla_4\rangle$
			
			\item $\langle \nabla_3+\nabla_5, \nabla_1, \nabla_4\rangle$
			
			\item $\langle \nabla_3, \nabla_2,\nabla_4 \rangle$

			\item $\langle \nabla_1 +\nabla_3, \nabla_2,\nabla_4 \rangle$			
			
			\item $\langle \nabla_5, \nabla_2,\nabla_4 \rangle$
			
			\item $\langle \nabla_1+\nabla_5, \nabla_2,\nabla_4 \rangle$
			
			\item $\langle \nabla_1+\nabla_5,\nabla_3, \nabla_4 \rangle$
			
			\item [$\mathcal P_{10}=$] $\langle \nabla_1+\nabla_5, \nabla_1+\nabla_3,\nabla_4 \rangle$
			
			\item [$\mathcal P_{11}=$] $\langle \nabla_2, \nabla_1,\nabla_3+\nabla_4 \rangle$
			
		\end{enumerate}
		
	\end{multicols}
\end{lemma}

\begin{proof}
	First, we observe that any orbit is generated by a subspace contained in one of the sets below:
	
	$S_1=\{\langle \alpha_2\nabla_2+\alpha_3\nabla_3+\alpha_5\nabla_5, \nabla_1, \nabla_4\rangle \,|\, \alpha_i\in \mathbb{C} \text{ are not all zero }\}$

	$S_2=\{\langle \alpha_1\nabla_1+\alpha_3\nabla_3+\alpha_5\nabla_5, \nabla_2, \nabla_4\rangle \,|\, \alpha_i\in \mathbb{C} \text{ are not all zero }\}$
	
	$S_3=\{\langle \alpha_1\nabla_1+\alpha_2\nabla_2+\alpha_5\nabla_5, \nabla_3, \nabla_4\rangle \,|\, \alpha_i\in \mathbb{C} \text{ are not all zero }\}$
	
	$S_4=\{\langle \alpha_1\nabla_1+\alpha_2\nabla_2+\alpha_5\nabla_5, \nabla_1+\nabla_3, \nabla_4\rangle \,|\, \alpha_i\in \mathbb{C} \text{ are not all zero }\}$
	
	$S_5=\{\langle \alpha_1\nabla_1+\alpha_2\nabla_2+\alpha_3\nabla_3, \nabla_5, \nabla_4\rangle \,|\, \alpha_i\in \mathbb{C} \text{ are not all zero }\}$
	
	$S_6=\{\langle \alpha_1\nabla_1+\alpha_2\nabla_2+\alpha_3\nabla_3, \nabla_1+\nabla_5, \nabla_4\rangle \,|\, \alpha_i\in \mathbb{C} \text{ are not all zero }\}$
	
	$S_7=\{\langle \alpha_1\nabla_1+\alpha_2\nabla_2+\alpha_3\nabla_3, \nabla_2+\nabla_5, \nabla_4\rangle \,|\, \alpha_i\in \mathbb{C} \text{ are not all zero }\}$
	
	$S_8=\{\langle \alpha_1\nabla_1+\alpha_2\nabla_2+\alpha_3\nabla_3, \nabla_1+\nabla_3+\nabla_5, \nabla_4\rangle \,|\, \alpha_i\in \mathbb{C} \text{ are not all zero }\}$
	
	$S_9=\{\langle\alpha_2\nabla_2+\alpha_3\nabla_3+ \alpha_5\nabla_5, \nabla_1, \nabla_3+\nabla_4\rangle \,|\, \alpha_i\in \mathbb{C} \text{ are not all zero }\}$
	
	$S_{10}=\{\langle\alpha_1\nabla_1+\alpha_3\nabla_3+ \alpha_5\nabla_5, \nabla_2, \nabla_3+\nabla_4\rangle \,|\, \alpha_i\in \mathbb{C} \text{ are not all zero }\}$

	In order to prove the result, we need to show that all elements in $S_1\cup \cdots \cup S_{10} $ lie in one of the orbits $\mathcal P_1, \mathcal P_{2}\dots, \mathcal P_{11}$, and that these orbits are pairwise distinct.

\begin{enumerate}
	\item  $\theta_2=\nabla_1, \theta_3=\nabla_4$.
	
	\begin{enumerate}
	\item Applying the automorphism $\phi$ in $\mathcal P_1=\langle \nabla_2, \nabla_1, \nabla_4 \rangle$ with $x=1,y=0,z=0,w=1,t=0,p=0$, we obtain all elements in $S_1$ with $\alpha_5=\alpha_3=0$, $\alpha_2\neq 0$.
	
	\item Applying the automorphism $\phi$ in $\mathcal P_2=\langle \nabla_3, \nabla_1, \nabla_4 \rangle$ with $x=1,y=0,z=\frac{1}{2}\alpha_{2},w=1,t=0,p=0$, we obtain all elements in $S_1$ with $\alpha_5=0$, $\alpha_3\neq 0$. Simple computations show that this orbit does not contain $\mathcal P_1$.	

	\item \label{III} Applying the automorphism $\phi$ in $\mathcal P_3=\langle \nabla_5, \nabla_1, \nabla_4 \rangle$ with $x=1,y=0,z=0,w=1,t=\alpha_{2},p=0$, we obtain all elements in $S_1$ with $\alpha_5=1$,  and $\alpha_3= 0$.	Moreover, any automorphism $\phi$ applied to this orbit does not contain $\mathcal P_1$ and $\mathcal P_2$. Indeed, let us consider two cases. If $y=0$, then one of the vectors is $\nabla_1$. One of the other vectors is $\nabla_4$ only if $p=0$ and the last vector has nonzero $\nabla_5$ component, and this case is settled. If $y\neq 0$, one of the vectors is $\nabla_4$ only if $p=t=0$. In this case one of the other vectors has nonzero $\nabla_5$ component. In both cases, they cannot be $\mathcal P_1$ and $\mathcal P_2$.

	\item Applying the automorphism $\phi$ in $\mathcal P_4=\langle \nabla_3+\nabla_5, \nabla_1, \nabla_4 \rangle$ with $x={\alpha_{{3}}}^{-1},y=0,z=\frac{1}{2}\alpha_{2},w=\alpha_{{3}},t=0,p=0$, we obtain all elements in $S_1$ with $\alpha_5=1$,  and $\alpha_3\neq 0$. Similar arguments as case (\ref{III}) show that by applying an automorphism $\phi$ to $\mathcal P_4$, so that we obtain vectors $\nabla_1$ and $\nabla_4$, the other vector has   nonzero $\nabla_3$ and $\nabla_5$ components. As a consequece, the orbits generated by $\mathcal P_1, \dots, \mathcal P_4$ are pairwise distinct.
\end{enumerate}

	\item  $\theta_2=\nabla_2, \theta_3=\nabla_4$.	
	
\begin{enumerate}	
	\item $\mathcal P_1$ contains all elements in $S_2$ with $\alpha_5=\alpha_3=0$, and no other elements of $S_2$.
	
	\item\label{6} Applying an automorphism $\phi$ in $\mathcal P_5=\langle \nabla_3, \nabla_2,\nabla_4 \rangle$, we obtain a subspace with $\nabla_4$ as one of the generating vectors only if $y=0$ and $t=\frac{pz}{w}$. In this case, the resulting subspace is $\langle -\beta_1^2\nabla_1+\nabla_3, \beta_1\nabla_1+\nabla_2, \nabla_4 \rangle$. This lies in $S_2$ only if $z=0$, i.e., when the resulting subspace is $\mathcal P_5$ itself. In partiular, the orbit generated by $\mathcal P_5$ does not contain $\mathcal P_1, \dots, \mathcal P_4$. 
	
	\item Applying an automorphism $\phi$ in $\mathcal P_6=\langle \nabla_1+\nabla_3, \nabla_2,\nabla_4 \rangle$, we obtain a subspace with $\nabla_4$ as one of the generating vectors only if $y=0$ and $t=\frac{pz}{w}$. In this case, the resulting subspace is $\langle \alpha_1\nabla_1+\nabla_3, \beta_1\nabla_1+\nabla_2, \nabla_4 \rangle$, with ${{\beta_1}^{2}+\alpha_1}\neq0$. In particular, this orbit contains all elements of $S_2$ with $\alpha_5=0$, $\alpha_3\neq 0$, $\alpha_1\neq0$ and no other element of $S_1\cup S_2$.
	
	\item Applying an automorphism $\phi$ in $\mathcal P_7=\langle \nabla_5, \nabla_2,\nabla_4 \rangle$ with $x=1,y=0,z=0,w=1,t=0,p=\alpha_{{3}}$, we obtain any element of $S_2$ with $\alpha_5=1$, $\alpha_1=0$.
	Moreover, applying an automorphism $\phi$ with $y=0$, we obtain a subspace in $S_1\cup S_2$ only if $t=\frac{pz}{w}$. The resulting subspace lies in $S_2$ only if   $z=0$. And the resulting subspaces are of the form $\langle \alpha_3 \nabla_3+\nabla_5, \nabla_2,\nabla_4 \rangle$, with $\alpha_3\in \mathbb{C}$. Similarly, if $w=0$, we obtain the same subspaces (with the same arguments). If $y\neq 0$ and $w\neq 0$, one of the generating vectors is $\nabla_4$ only if $p=t=0$. In this case, the resulting subspace does not lie in $S_1\cup S_2$.
	
	\item Applying an automorphism $\phi$ in $\mathcal P_8=\langle \nabla_1+ \nabla_5, \nabla_2,\nabla_4 \rangle$ with $x=\alpha_1,y=0,z=0,w=1,t=0,p=\alpha_1\alpha_{{3}}$, we obtain any element of $S_2$ with $\alpha_5=1$, $\alpha_1\neq 0$. Similar arguments as the above case show that the orbit generated by $\mathcal P_8$ does not contain $\mathcal P_1, \dots, \mathcal P_7$. In particular, the subspaces  $\mathcal P_1, \dots, \mathcal P_8$ generate pairwise distinct orbits containing $S_1\cup S_2$.
	
\end{enumerate}

	\item $\theta_2=\nabla_3, \theta_3=\nabla_4$

\begin{enumerate}

	\item $\mathcal P_2$ contains all elements in $S_3$ with $\alpha_5=\alpha_2=0$, and no other elements of $S_3$.
	
	\item $\mathcal P_5$ contains all elements in $S_3$ with $\alpha_1=\alpha_5=0$. In a similar way as proved in $(\ref{6})$, this orbit does not contain other elements of $S_3$.
	
	\item Applying the automorphism $\phi$ in $\mathcal P_6$, with $x=\sqrt \alpha_1,y=0,z=0,w=1,t=0,p=0$, we obtain all elements of $S_3$ with $\alpha_1\neq 0$, $\alpha_2\neq 0$, $\alpha_5=0$.
	
	\item Applying the automorphism $\phi$ in $\mathcal P_3$, with $y=0$, the resulting subspace does not lie in $S_3$. If $y\neq 0$, the resulting subspace lies in $S_3$ only if $p=t=x=0$, and we obtain the subspace $\langle \nabla_5,\nabla_3,\nabla_4\rangle$, i.e., the element of $S_3$ with $\alpha_5=1$, $\alpha_2=\alpha_1=0$.
	
	\item Applying the automorphism $\phi$ in $\mathcal P_9=\langle \nabla_1+\nabla_5,\nabla_3, \nabla_4 \rangle$, with 
	$x=\alpha_{1},y=\frac{\alpha_2}{3},z=0,w=1,t=\frac{\alpha_1\alpha_2}{3},p=\frac{2 \alpha_2^2}{9}$, we obtain all elements in $S_3$ with $\alpha_5=1$, $\alpha_1\neq 0$.
	Moreover, by applying an arbitrary automorphism $\phi$ in $\mathcal P_9$, if $y=0$, the resulting subspace lies in $S_1\cup S_2\cup S_3$ only if $p=t=0$, and in this case it is equal to $\langle \alpha_1\nabla_1+\nabla_5,\nabla_3,\nabla_4 \rangle$, with $\alpha_1\neq0$. If $w=0$, the resulting subspace does not lie in $S_1\cup S_2\cup S_3$. If $y\neq 0$ and $w\neq 0$, the resulting subspace lies in $S_1\cup S_2\cup S_3$ only if $z=0$, $p=\frac{2y^2}{w}$ and $t=\frac{xy}{w}$ and this yields all subspaces of $S_3$ with $\alpha_1\neq0$, and $\alpha_5=1$.
	
	\item Applying the automorphism $\phi$ in $\mathcal P_{4}$ with $x=0,y=\frac{1}{\alpha_2},z=z,w=w,t=-z,p=-2w $, we obtain any element in $S_3$ with $\alpha_1=0$, $\alpha_2\neq 0$ and $\alpha_5=1$. And all elements of $S_1\cup S_2\cup S_3$ are in the distinct orbits generated by the subspaces $\mathcal P_1, \dots, \mathcal P_9$.
	

\end{enumerate}

	\item $\theta_2=\nabla_1+\nabla_3, \theta_3=\nabla_4$
	
\begin{enumerate} 

	\item $\mathcal P_2$ is an element of $S_4$ with $\alpha_2=\alpha_5=0$.
	
	\item By applying the automorphism $\phi$ in $\mathcal P_{6}=\langle \nabla_1+\nabla_3,\nabla_2,\nabla_4 \rangle$ with $x=\sqrt {\alpha_1^{2}+1},y=0,z=\alpha_{1},w=1,t=0,p=0,$ we obtain any element in $S_4$ with $\alpha_5=0$, $\alpha_2=1$ and $\alpha_1^2\neq -1$.
	
	\item By applying the automorphism $\phi$ in $\mathcal P_5$ with $x=w=1$, $y=p=t=0$, we obtain $\langle i\nabla_1+\nabla_2,\nabla_1+\nabla_3,\nabla_4 \rangle$, if $z=i$ and $\langle -i\nabla_1+\nabla_2,\nabla_1+\nabla_3,\nabla_4 \rangle$, if $z=-i$.
	
	
	\item An element of $S_4$ with $\alpha_1=\alpha_2=0$ lies in the orbit generated by $\mathcal P_7$. Indeed, one just need to apply the automorphism $\phi$ with $x=i,y=1,z=-i,w=1,t=0,p=0$ to the later to obtain the former.
	
	\item By applying the automorphism $\phi$ in $\mathcal P_{10} =\langle \nabla_1+\nabla_5, \nabla_1+\nabla_3,\nabla_4 \rangle$, with $x={\frac {\sqrt {3}}{2\alpha_2}},y=\frac{1}{2\alpha_2},z=-\frac{1}{2\alpha_2},w={\frac {\sqrt {3}}{2\alpha_2}},t=\frac{1}{4\alpha_2},p={\frac {\sqrt {3}}{4\alpha_2}}
	$, we obtain any element in $S_4$ with $\alpha_5=1$, $\alpha_2\neq 0$ and $\alpha_1=0$.
	
	By considering $x=-\frac{1}{2\alpha_1},y=-\frac {\sqrt {3}}{2\alpha_1},z={\frac {\sqrt{3}}{2\alpha_1}},w=-\frac{1}{2\alpha_1},t={\frac{\sqrt{3}}{4\alpha_1}},p=-\frac{3}{4\alpha_1}$, we obtain any element in $S_4$ with $\alpha_5=1$, $\alpha_2=0$ and $\alpha_1\neq 0$.
	
	Let us now show that this orbit also contains all elements of $S_4$ with $\alpha_5=1$ and $\alpha_2\neq 0$ $\alpha_1\neq 0$, provided that $\alpha_1^2+\alpha_2^2\neq0$. Indeed by applying $\phi$ in $\mathcal P_{10}$ with variables defined inductively by $t=\frac{y(w^2-y^2)}{y^2+w^2}$, $p=\frac{2wy^2}{y^2+w^2}$,    $z=-y$,  $x=w$, we obtain the subspaces $\langle \alpha_1\nabla_1+\alpha_2\nabla_2+\nabla_5,\nabla_1+\nabla_3,\nabla_4\rangle$, where 
	$\alpha_1=\frac{w(w^2-3y^2)}{(w^2+y^2)^2}$ and $\alpha_2=\frac{y(3w^2-y^2)}{(w^2+y^2)^2}$. Now Lemma \ref{solution} shows that if $\alpha_1^2+\alpha_2^2\neq0$, we can find $y,w\in \mathbb{C}$ satisfying the above equations, and this case is settled.
	
	Now we show that the orbit generated by $\mathcal P_{10}$ does not contain $\mathcal P_1\cup \dots \cup P_9$. To that it is enough to show that this orbit does not intersect $S_1\cup S_2\cup S_3$.
	
	We consider different cases. By applying an automorphism $\phi$ in $\mathcal P_{10}$ with $y=0$, the resulting subspace contains $\nabla_4$ only if $p=t=0$. In this case, the vectors $\nabla_1$, $\nabla_2$ and $\nabla_3$ do not lie in the resulting subspace. As a consequence, the resulting subspace does not lie in $S_1\cup S_2\cup S_3$. Similar arguments apply to the case $w=0$ and to the case  $yw\neq0$, $y^2+w^2\neq0$.
	
	If $yw\neq 0$, $y^2+w^2=0$, the resulting subspace does not contain $\nabla_4$. In particular, it does lie in $S_1\cup S_2\cup S_3$.
	
	\item By applying the automorphism $\phi$ in $\mathcal P_8$, with $x=\alpha_{1},y=i\alpha_{1},z=i,w=1,t=i{\alpha_{1}}^{2},p=-{\alpha_{1}}^{2}$, we obtain any element of $S_4$ with $\alpha_5=1$ and  $\alpha_2= i\alpha_1\neq0$, and with $x=\alpha_{1},y=-i\alpha_{1},z=-i,w=1,t=-i{\alpha_{1}}^{2},p=-{\alpha_{1}}^{2}$, we obtain any element of $S_4$ with $\alpha_5=1$, and $\alpha_2= -i\alpha_1\neq0$.
	
	
\end{enumerate}
	Up to here, we have that any element of $S_1\cup \dots \cup S_4$ is in one and only one of the distinct orbits $\mathcal P_1, \dots, \mathcal P_{10}$.
	
\item $\theta_2=\nabla_5, \theta_3=\nabla_4$

\begin{enumerate} 
	\item By applying $\phi$ in $\mathcal P_7$ with $x=w=1$, $z=\alpha_1$ and $y=t=p=0$, we obtain all elements of $S_5$ with $\alpha_3=0$ and $\alpha_2=1$. The element with $\alpha_3=0$ and $\alpha_2=0$ lies (trivially) in the orbit generated by $\mathcal P_3$.
	
	\item By applying $\phi$ in $\mathcal P_7$ with $x=\alpha_{2}-\sqrt {{\alpha_{2}}^{2}-4\,\alpha_{1}},y=2,z=\alpha_{2}+\sqrt {{\alpha_{2}}^{2}-4\,\alpha_{1}},w=2,t=0,p=0$,
	 we obtain all elements of $S_5$ with $\alpha_3=1$ and $4\alpha_1-\alpha_2^2\neq 0$. 
	
	\item By applying $\phi$ in $\mathcal P_3$ with $x=\alpha_2$, $y=2$, $z=1$, $w=t=p=0$, we obtain all elements of $S_5$ with $\alpha_3=1$ and $4\alpha_1-\alpha_2^2=0$. 
	
\end{enumerate}
	 And we have obtained all elements of $S_5$.
	 
	\item $\theta_2=\nabla_1+\nabla_5, \theta_3=\nabla_4$
	
	\begin{enumerate} 
		\item By applying $\phi$ in $\mathcal P_8$ with $x=w=1$, $z=\alpha_1$ and $y=p=t=0$, we obtain all elements in $\mathcal S_6$ with $\alpha_3=0$ and $\alpha_2=1$. The element with $\alpha_3=0$ and $\alpha_2=0$ lies trivially in $\mathcal P_3$.
		
		\item By applying $\phi$ in $\mathcal P_9$ with $x=4$, $z=\alpha_2$, $w=2$ and $y=t=p=0$, we obtain all elements in $S_6$ with $\alpha_3=1$ and $4\alpha_1-\alpha_2 ^2=0$.
		
		\item By applying $\phi$ in $\mathcal P_{10}$ with $x=\frac{4\alpha_1-\alpha_2^2}{4},y=0,z=\frac{\alpha_2\sqrt{4\alpha_1-\alpha_2^2}}{4},w=\frac{\sqrt{4\alpha_1-\alpha_2^2}}{2},t=0,p=0$, we obtain all elements of $S_6$ with $\alpha_3=1$ and $4\alpha_1-\alpha_2^2\neq0$.

	\end{enumerate}	
	
		\item $\theta_2=\nabla_2+\nabla_5, \theta_3=\nabla_4$
 	
	\begin{enumerate} 
	
\item 
{ By applying $\phi$ in $\mathcal P_{10}$ with $x=0, y=-1, z=\sqrt {\alpha_1}, w=0, t=\sqrt {\alpha_1}, p=0$}, we obtain any element in $S_7$ with $\alpha_3=1$, $\alpha_2=0$ and $\alpha_1\neq 0$. The element of $S_7$ with $\alpha_2=\alpha_1=0$ lies in $\mathcal P_7$.
		
		\item By applying the automorphism $\phi$ in $\mathcal P_8$ with $x=0,y=1,z=\alpha_{2},w=1,t=0,p=1$, we obtain all elements of $S_7$ with $\alpha_3=1$, $\alpha_2\neq0$ and $\alpha_1=0$.
		
		\item We now apply an automorphism $\phi$ in $\mathcal P_{10}$ with $yw\neq 0$ and $y^2+w^2\neq0$. The resulting subspace lies in $S_7$ only if $p=\frac{2wy^2}{y^2+w^2}$, $t=\frac{y(wx+yz)}{y^2+w^2}$, $z=\frac{x(y^2-w^2)}{2yw}$. Moreover, $y$ and $w$ must satisfy $\frac{(3w^2-y^2)y}{y^2+w^2}=1$. With this condition, the resulting subspace equals $\langle \alpha_1\nabla_1 + \alpha_2 \nabla_2 , \nabla_2+\nabla_5, \nabla_4\rangle$. Here, $\alpha_2= {\frac {x (3{y}^{2}-{w}^{2})}{y ({w}^{2}+{y}^{2})}}$ and $\alpha_1={\frac{({w}^{2}+{y}^{2})x^2}{4w^2y^2}}$. In particular, we obtain $x={\frac {\alpha_{2}y \left( {w}^{2}+{y}^{2} \right) }{3{y}^{2
		}-{w}^{2}}}$, and that $y$ and $w$ must satisfy the following system of equations:
	
	\[\left\{\begin{array}{rcl} {\frac {({w}^{2}+{y}^{2})^3 \alpha_{2}^{2}}{4 \left( {w}^{2
			}-3\,{y}^{2} \right) ^{2}{w}^{2}}} & = & \alpha_{1}\\
		\frac{(3w^2-y^2)y}{y^2+w^2} & = & 1
	\end{array}\right.	\]
	
	Now, from Lemma \ref{eq2}, if $\alpha_1\alpha_2\neq0$, the above system of equations has a solution if and only  $4\alpha_1-\alpha_2^2\neq0$. This means we have obtained all elements from $S_7$ with $\alpha_3=1$, $\alpha_1\alpha_2\neq0$ and $4\alpha_1-\alpha_2^2\neq0$.
	
	\item By applying the automorphism $\phi$ in $\mathcal P_9$ with  $x=\frac{-\alpha_2}{3}$, $y=\frac{1}{3}$, $z=\frac{\alpha_2}{2}$, $w=1$, $t=\frac{-\alpha_2}{18}$, $p=\frac{2}{9}$,   we obtain all elements of $S_7$ with $\alpha_3=1$, $\alpha_2\neq0$ and $4\alpha_1=\alpha_2^2$.
	
	\item By applying the automorphism $\phi$ in $\mathcal P_8$ with $x=-\alpha_1$, $y=0$, $z=\alpha_1$, $w=1$, $t=0$, $p=0$, we obtain any element of $S_7$ with $\alpha_3=0$, $\alpha_2=1$ and $\alpha_1\neq0$. The element of $S_7$ with $\alpha_3=0$ and $\alpha_1=0$ lies trivially in $\mathcal P_7$.
	
	\item The element of $S_7$ with $\alpha_3=0$, $\alpha_2=0$, lies in $\mathcal P_3$.
	\end{enumerate}

	\item $\theta_2=\nabla_1+\nabla_3+\nabla_5, \theta_3=\nabla_4$

\begin{enumerate} 
		\item The element of $S_8$ with $\alpha_3=0$, $\alpha_2=0$ lies trivially in $\mathcal P_4$.
		
		\item By applying $\phi$ in $\mathcal P_8$ with $x={\alpha_{1}}^{2}+1,y=0,z=\alpha_{1},w=1,t= \left( {\alpha_{1}}^{2}+1
		\right) \alpha_{1},p={\alpha_{1}}^{2}+1$,  we obtain any element of $S_8$ with $\alpha_3=0$, $\alpha_2=1$ and $\alpha_1^2+1\neq0$. 
		
		\item By applying $\phi$ in $\mathcal P_7$ with $x=1,y=0,z=\pm i,w=1,t=\pm i,p=1$, we obtain all elements of $S_8$ with $\alpha_3=0$, $\alpha_2=1$ and $\alpha_1^2+1=0$.
		
		
		\item\label{d} By applying $\phi$ in $\mathcal P_4$ with $x=-\frac{1}{2}$, $y=\pm\frac{i}{2}$, $z=1$, $w=0$, $t=-1$, $p=0$, we obtain any element of $S_8$ with $\alpha_3=1$, $\alpha_2^2+4=0$ and $4\alpha_1-\alpha_2^2=0$.
		
		\item By applying $\phi$ in $\mathcal P_9$ with $x=1+\frac{\alpha_2^2}{12}$, $y= \frac{-\alpha_2}{3}$, $z=\frac{\alpha_2}{2}$, $w=1$, $t=\frac{\alpha_2(\alpha_2^2-12)}{36}$, $p=\frac{2\alpha_2^2}{9}$, we obtain all elements in $S_8$ with $\alpha_3=1$, $\alpha_2^2+4\neq 0$ and $4\alpha_1-\alpha_2^2=0$.
		

		\item By applying $\phi$ in $\mathcal P_{10}$ with $x=\frac{\alpha_1}{1-\alpha_1}$, $w=\frac{\sqrt{\alpha_1}}{\alpha_1-1}$, $y=z=p=t=0$, we obtain all elements of $S_8$ with $\alpha_3=1$, $\alpha_2=0$ and $\alpha_1\not\in \{0,1\}$. The case $\alpha_1=0$ was treated in the previous case.
		
		\item\label{g} By applying $\phi$ in $\mathcal P_7$ with $x=-i$, $y=1$, $z=i$, $w=1$ and $p=t=0$, we obtain the element of $S_8$ with $\alpha_3=1$, $\alpha_2=0$ and $\alpha_1=1$.
		

		\item By applying $\phi$ in $\mathcal P_{10}$ with $x=\frac{1}{2},y=\frac{1}{\alpha_2},z=\frac{\sqrt{\alpha_2^{2}+4}}{2\alpha_{2}},w=0,t=\frac{\sqrt{\alpha_2^{2}+4}}{2\alpha_{2}},p=0
		$, we obtain all elements of $S_8$ with $\alpha_3=1$, $\alpha_2\neq0$ and $\alpha_1=\frac{\alpha_2^2}{2}+1$.

		\item By applying $\phi$ in $\mathcal P_{10}$ with		
		$x={\frac{({k}^{2}-3)({k}^{2}\alpha_{2}^{2}+{k}^{2}\alpha_{1}-{k}^{2}+\alpha_{2}^{2}-3\alpha_{1}+3) {k}^{2}}{({k}^{2}+1 ) ^{2} ( 3{k}^{2}-1 ) \alpha_{2}^{2}}}$,
		$y={\frac{{k}^{2}({k}^{2}-3)}{\alpha_{2}({k}^{2}+1)^{2}}}$, 
		$z={\frac{k(3-{k}^{2})(-{k}^{4}\alpha_{2}^{2}+2{k}^{4}\alpha_{1}-2\,{k}^{4}-6{k}^{2}\alpha_{1}+6{k}^{2}+\alpha_{2}^{2})}{2({k}^{2}+1)^{2}(3{k}^{2}-1)\alpha_{2}^{2}}}$, 
		$w={\frac{k({k}^{2}-3)}{\alpha_{2}({k}^{2}+1)^{2}}}$,
		$t={\frac{(3-{k}^{2})(-{k}^{4}\alpha_{2}^{2}+2{k}^{4}\alpha_{1}-2{k}^{4}-2{k}^{2}\alpha_2^{2}-8{k}^{2}\alpha_{1}+8{k}^{2}-\alpha_{2}^{2}+6\alpha_{1}-6){k}^{3}}{2\alpha_{2}^{2}(3{k}^{2}-1)({k}^{2}+1)^{3}}}$,
		$p={\frac{2{k}^{3}({k}^{2}-3)}{\alpha_{2}({k}^{2}+1)^{3}}}$, we must have $\det \phi=		
		{\frac { ( {k}^{2}-3 ) ^{3} ( -\alpha_{2}^{2}+2\alpha_{1}-2){k}^{3}}{ 2(3{k}^{2}-1)({k}^{2}+1)^{3}\alpha_{2}^{3}}}		
		\neq0$. In this case, we obtain the subspaces of the form \[\langle f\nabla_1+\alpha_2\nabla_2+\nabla_3, g\nabla_1+\nabla_3+\nabla_5,\nabla_4\rangle,\] where $f$ and $g$ are rational functions in $k$ satisfying $f-\alpha_1=g-1$.
		
		Once we find $k\in \mathbb{C}$, such that $k\not\in \{0,\pm i, \pm \sqrt{3}, \pm \frac{1}{\sqrt{3}}\}$ and $g-1=0$, we obtain $g=1$ and $f=\alpha_1$, i.e., the resulting subspace is 
		\[\langle \alpha_1\nabla_1 + \alpha_2\nabla_2+\nabla_3, \nabla_1+\nabla_3 + \nabla_5,\nabla_4 \rangle. \]
		We will show now that it is possible to find such $k$ if $\alpha_2\neq0$, $2\alpha_1-\alpha_2^2-2\neq 0$, $4\alpha_1-\alpha_2\neq0$ and $\alpha_2^2+(\alpha_1-1)^2\neq0$, i.e., we will obtain all elements of $S_8$ with the above conditions on $\alpha_1$ and $\alpha_2$.
		
		The numerator of the polynomial $g-1$ is the nonzero polynomial
		\begin{eqnarray*}
		G(k) & = & (2\alpha_1-\alpha_2^2-2)^2k^6+(3\alpha_2^4 -12\alpha_1\alpha_2^2 -24\alpha_1^2-24\alpha_2^2+48\alpha_1-24)k^4 + \\ & + &  (3\alpha_2^4 -12\alpha_1\alpha_2^2 +36\alpha_1^2+36\alpha_2^2-72\alpha_1+36)k^2+\alpha_2^2(\alpha_2^2-4\alpha_1).
		\end{eqnarray*}
	
		To finish this case, it is enough to show that $0,\pm i, \pm \sqrt{3}, \pm \frac{1}{\sqrt{3}}$ are not roots of $G$.
		
		First we observe that $0$ is not a root of $G$ since $G(0)=\alpha_2^2(\alpha_2^2-4\alpha_1)$  which is nonzero by hypothesis.
		
		The same with $k=\pm \sqrt{3}$, since $G(\pm \sqrt{3})=64\alpha_2^2(\alpha_2^2-4\alpha_1)$.
		
		Also, $k= \pm \frac{1}{\sqrt{3}}$ are not roots of $G$, since 
		$G(\pm \frac{1}{\sqrt{3}})=\frac{64}{27}(2\alpha_1-\alpha_2^2-2)^2.$
		
		Finally since $G(\pm i)=-64(\alpha_2^2+(\alpha_1-1)^2)$, we obtain that $\pm i$ are not  roots of $G$, and this case is settled.

		\item By applying $\phi$ in $\mathcal P_8$ with $x=-\alpha_{1}+1,y=\pm i(1-\alpha_{1}) ,z=\pm i\alpha_{1},w=1,t=\pm i(\alpha_{1}-1) ^{2},p=- (\alpha_{1}-1 ) ^{2}$ we obtain all elements of $S_8$ with $\alpha_3=1$, $\alpha_2\neq0$, $2\alpha_1-\alpha_2^2-2\neq 0$, $4\alpha_1-\alpha_2\neq0$ and $\alpha_2^2+(\alpha_1-1)^2=0$, for $\alpha_1^2\neq 1$. If $\alpha_1=1$, we have $\alpha_2=0$, and this case have already been considered in (\ref{g}). The case with $\alpha_1=-1$ was dealt with in (\ref{d}).
		
\end{enumerate}

\item $\theta_2=\nabla_1$, $\theta_3=\nabla_3+\nabla_4$.

	\begin{enumerate}
		\item Let $\mathcal P_{11}=\langle \nabla_2,\nabla_1,\nabla_3+\nabla_4\rangle$. This orbit contains the element of $S_9$ with $\alpha_5=\alpha_3=0$. Moreover, the orbit generated by $\mathcal P_{11}$ does not contain any other element of $S_1\cup \cdots \cup S_8$. Indeed, by applying an arbitrary automorphism $\phi$ in $\mathcal P_{11}$, the resulting subspace does not contain the vector $\nabla_4$. As a consequence, the orbits generated by $\mathcal P_{1}, \dots, \mathcal P_{11}$ are pairwise distinct.
		
		\item By applying $\phi$ in $\mathcal{P}_2$ with $x=1, y=0, z=\alpha_2, w=2, t=0, p=-2\alpha_2$, we obtain all elements of $S_9$ with $\alpha_5=0$ and $\alpha_3=1$. 
		
		
		\item By applying $\phi$ in $\mathcal P_9$ with $x=3\alpha_{3},y=-9,z=3,w=0,t=6\alpha_3^{2}+27\alpha_{2},p=18\alpha_{3}$, we obtain all elements of  $S_9$ with $\alpha_5=1$.
\end{enumerate}

	\item $\theta_2=\nabla_2$, $\theta_3=\nabla_3+\nabla_4$.

	\begin{enumerate}

		\item The element of $S_{10}$ with $\alpha_5=\alpha_3=0$ is $\mathcal P_{11}$.
		
		\item The element of $S_{10}$ with $\alpha_5=\alpha_1=0$ is $\mathcal P_{5}$.
		
		\item By applying $\phi$ in $\mathcal P_6$ with $x=\sqrt{\alpha_1}$, $y=z=0$, $w=1$, $t=-\sqrt{\alpha_1}^3$, $p=0$, we obtain all elements of $S_{10}$ with $\alpha_5=0$, $\alpha_3=1$ and $\alpha_1\neq0$.
		
		
		\item By applying $\phi$ in $\mathcal P_8$ with $x=0, y=-1,z=1,w=0,t=0,p=\alpha_3$, we obtain all elements of $S_{10}$ with $\alpha_5=1$ and $\alpha_1=0$.
		
		\item By applying $\phi$ in $\mathcal P_{10}$ with $x=-iz, y=-4z^2,w=4iz^2, t=\frac{z}{2}, p=2z^2(4\alpha_3z+i)$, where $z$ is a cubic root of  $\frac{i}{64\alpha_1}$, we obtain all elements of $S_{10}$ with $\alpha_5=1$ and $\alpha_1\neq 0$.
	\end{enumerate}
\end{enumerate}

All cases have been considered and the lemma is proved.
\end{proof}
 
\subsection{$4$-dimensional central extensions of $\mathfrak{N}_1$}

We may assume that a $4$-dimensional subspace is generated by
$\theta_1, \theta_2$, $\theta_3$ and $\theta_4$,  where $\langle \theta_2, \theta_3, \theta_4 \rangle\in \{\mathcal P_1, \dots, \mathcal P_{11}\}$.

\begin{lemma}
	The following subspaces generate all pairwise distinct orbits.
	
	\begin{multicols}{2}
		\begin{enumerate}[$\mathcal{Q}_1=$]
			\item $\langle \nabla_3, \nabla_2, \nabla_1, \nabla_4\rangle$
			
			\item $\langle \nabla_5, \nabla_2, \nabla_1, \nabla_4\rangle$
			
			\item $\langle \nabla_5, \nabla_3, \nabla_1, \nabla_4 \rangle$
			
			\item $\langle \nabla_1+\nabla_5, \nabla_3, \nabla_2, \nabla_4 \rangle$
			
		\end{enumerate}
		
	\end{multicols}
\end{lemma}

\begin{proof}

First, we observe that any orbit is generated by a subspace contained in one of the sets below:

$S_1=\{\langle \alpha_3\nabla_3+\alpha_5\nabla_5, \nabla_2, \nabla_1, \nabla_4\rangle \,|\, \alpha_i\in \mathbb{C} \text{ are not all zero }\}$

$S_2=\{\langle \alpha_2\nabla_2+\alpha_5\nabla_5,\nabla_3, \nabla_1, \nabla_4\rangle \,|\, \alpha_i\in \mathbb{C} \text{ are not all zero }\}$

$S_3=\{\langle \alpha_2\nabla_2+\alpha_3\nabla_3,\nabla_5, \nabla_1, \nabla_4\rangle \,|\, \alpha_i\in \mathbb{C} \text{ are not all zero }\}$

$S_4=\{\langle \alpha_2\nabla_2+\alpha_3\nabla_3,\nabla_3+\nabla_5, \nabla_1, \nabla_4\rangle \,|\, \alpha_i\in \mathbb{C} \text{ are not all zero }\}$

$S_5=\{\langle \alpha_1\nabla_1+\alpha_5\nabla_5,\nabla_3, \nabla_2, \nabla_4\rangle \,|\, \alpha_i\in \mathbb{C} \text{ are not all zero }\}$

$S_6=\{\langle \alpha_1\nabla_1+\alpha_5\nabla_5,\nabla_1+\nabla_3, \nabla_2, \nabla_4\rangle \,|\, \alpha_i\in \mathbb{C} \text{ are not all zero }\}$

$S_7=\{\langle \alpha_1\nabla_1+\alpha_3\nabla_3,\nabla_5, \nabla_2, \nabla_4\rangle \,|\, \alpha_i\in \mathbb{C} \text{ are not all zero }\}$

$S_8=\{\langle \alpha_1\nabla_1+\alpha_3\nabla_3,\nabla_1+\nabla_5, \nabla_2, \nabla_4\rangle \,|\, \alpha_i\in \mathbb{C} \text{ are not all zero }\}$

$S_9=\{\langle \alpha_1\nabla_1+\alpha_2\nabla_2,\nabla_1+\nabla_5, \nabla_3, \nabla_4\rangle \,|\, \alpha_i\in \mathbb{C} \text{ are not all zero }\}$

$S_{10}=\{\langle \alpha_1\nabla_1+\alpha_2\nabla_2,\nabla_1+\nabla_5, \nabla_1+\nabla_3, \nabla_4\rangle \,|\, \alpha_i\in \mathbb{C} \text{ are not all zero }\}$

$S_{11}=\{\langle \alpha_3\nabla_3+\alpha_5\nabla_5, \nabla_1, \nabla_2,\nabla_3+\nabla_4\rangle \,|\, \alpha_i\in \mathbb{C} \text{ are not all zero }\}$

In order to prove the result, we need to show that all elements in $S_1\cup \cdots \cup S_{11} $ lie in one of the orbits $\mathcal Q_1, \mathcal Q_{2}\dots, \mathcal Q_{?}$, and that these orbits are pairwise distinct.

\begin{enumerate}
	\item $\theta_2=\nabla_2, \theta_3=\nabla_1, \theta_4=\nabla_4$
	
	\begin{enumerate}
		\item The element of $S_1$ with $\alpha_5=0$ is $\mathcal Q_1=\langle \nabla_3, \nabla_2, \nabla_1, \nabla_4\rangle$.
		
		\item By applying $\phi$ in $\mathcal Q_2=\langle \nabla_5, \nabla_2, \nabla_1, \nabla_4\rangle$ with $x=1, y=0, z=0, w=1, t=0, p=\alpha_3$, we obtain all elements of $S_1$ with $\alpha_5=1$. 
		
		Moreover $\mathcal Q_1$ and $\mathcal Q_2$ generate distinct orbits. Indeed, by applying an arbitrary $\phi$ in $\mathcal Q_1$, to obtain the vector $\nabla_4$ we must have $y=0$. In this case, the resulting subspace is $\mathcal Q_1$ itself.
	\end{enumerate}
	
	\item $\theta_2= \nabla_3, \theta_3=\nabla_1, \theta_4=\nabla_4$

	\begin{enumerate}
		\item The element of $S_2$  with $\alpha_5=0$ is $\mathcal Q_1$. 
		
		\item By applying $\phi$ in $\mathcal Q_3=\langle \nabla_5, \nabla_3, \nabla_1, \nabla_4 \rangle$, with $x=0,y=1,z=1,w=0,t=-\alpha_2, p=0$, we obtain all elements of $S_2$ with $\alpha_5=1$.
		
		Moreover, the orbit generated by $\mathcal Q_3$ does not contain $\mathcal Q_1$ and $\mathcal Q_2$. Indeed, by applying $\phi$ in $\mathcal Q_3$ with $w=0$ or $y=0$, the resulting subspace does not contain $\nabla_2$, so it can be neither $\mathcal Q_1$ nor $\mathcal Q_2$. Similarly, if $yw\neq 0$ the resulting subspace does not contain $\nabla_1$, and the same conclusion holds. As a consequence, $\mathcal Q_1$, $\mathcal Q_2$ and $\mathcal Q_3$ generate distinct orbits.
	\end{enumerate}
		
	\item $\theta_2= \nabla_5, \theta_3=\nabla_1, \theta_4=\nabla_4$
	
	\begin{enumerate}
		\item The element of $S_3$ with $\alpha_3=0$ is $\mathcal Q_2$. 
		
		\item By applying $\phi$ in $\mathcal Q_3$ with $x=1,y=0,z=\frac{\alpha_2}{2}, w=1, t=p=0$, we obtain all elements of $S_3$ with $\alpha_3=1$.
	\end{enumerate}

	\item $\theta_2=\nabla_3+\nabla_5, \theta_3=\nabla_1, \theta_4=\nabla_4$
	
	\begin{enumerate}
		\item By applying $\phi$ in $\mathcal Q_2$ with $x=1,y=0,z=0,w=1,t=0,p=1$, we obtain the element of $S_4$ with $\alpha_3=0$.
		
		\item By applying $\phi$ in $\mathcal Q_3$ with $x=1,y=0,z=\alpha_2, w=2, t=-2\alpha_2, p=0$, we obtain all elements of $S_4$ with $\alpha_3=1$.
	\end{enumerate}

	\item $\theta_2=\nabla_3, \theta_3=\nabla_2, \theta_4=\nabla_4$

	\begin{enumerate}
		\item The element of $S_5$ with $\alpha_5=0$ is $\mathcal Q_1$.
		
		\item By applying $\phi$ in $\mathcal Q_2$ with $x=0,y=1,z=1,w=0,t=0,p=0$, we obtain the element of $S_5$ with $\alpha_1=0$.
		
		\item By applying $\phi$ in $\mathcal Q_4=\langle \nabla_1+\nabla_5, \nabla_3, \nabla_2, \nabla_4 \rangle$ with $x=\alpha_1, y=0,z=0,w=1,t=0,p=0$, we obtain all elements of $S_5$ with $\alpha_5=1$ and $\alpha_1\neq 0$.
	\end{enumerate}
	
	\item $\theta_2=\nabla_1+\nabla_3, \theta_3=\nabla_2, \theta_4=\nabla_4$
	
	\begin{enumerate}
		\item The element of $S_6$ with $\alpha_5=0$ is $\mathcal Q_1$.
		
		\item By applying $\phi$ in $\mathcal Q_3$ with $x=1,y=1,z=-1,w=1,t=0,p=-2\alpha_{1}$, we obtain all elements of $S_6$ with $\alpha_5=1$.
	\end{enumerate}
	
	\item $\theta_2=\nabla_5, \theta_3=\nabla_2, \theta_4=\nabla_4$
	
	\begin{enumerate}
		\item The element of $S_7$ with $\alpha_3=0$ is $\mathcal Q_2$.
		
		\item By applying $\phi$ in $\mathcal Q_2$ with $x=0,y=1,z=1,w=0,t=0,p=0$, we obtain the element of $S_7$ with $\alpha_1=0$.
		
		\item By applying $\phi$ in $\mathcal Q_3$ with $x=\sqrt{\alpha_1}, y=1, z=-\sqrt{\alpha_1},w=1, t=0, p=0$, we obtain all elements of $S_7$ with $\alpha_3=1$ and $\alpha_1\neq 0$.
	\end{enumerate}

	\item $\theta_2=\nabla_1+\nabla_5, \theta_3=\nabla_2, \theta_4=\nabla_4$

	\begin{enumerate}
		\item The element of $S_8$ with $\alpha_3=0$ is $\mathcal Q_2$.
		
		\item The element of $S_8$ with $\alpha_1=0$ is $\mathcal Q_4$.
		
		\item By applying $\phi$ in $\mathcal Q_3$ with $x=-i\sqrt {\alpha_1},y=i,z=\sqrt {\alpha_1},w=1,t=0,p=\frac {2i}{\sqrt {\alpha_1}}$, we obtain all elements of $S_8$ with $\alpha_3=1$ and $\alpha_1\neq0$.
	\end{enumerate}

	\item $\theta_2=\nabla_1+\nabla_5, \theta_3=\nabla_3, \theta_4=\nabla_4$
		
	\begin{enumerate}
		\item The element with $\alpha_2=0$ in $S_9$ is $\mathcal Q_3$.
		
		\item The element with $\alpha_1=0$ in $S_9$ is $\mathcal Q_4$.
		
		\item By applying $\phi$ in $\mathcal Q_3$ with $x=2\alpha_1, y=1, z=0, w=1, t=-2, p=-\frac{2}{\alpha_1}$, we obtain any element of $S_9$ with $\alpha_2=1$ and $\alpha_1\neq 0$. 
	\end{enumerate}	

	\item $\theta_2=\nabla_1+\nabla_5, \theta_3=\nabla_1+\nabla_3, \theta_4=\nabla_4$

	\begin{enumerate}
		\item The element of $S_{10}$ with $\alpha_2=0$ is $\mathcal Q_3$.
		
	\item By applying $\phi$ in $\mathcal Q_3$ with $x=-1$, $y=z=\alpha_1+\sqrt{\alpha_1^2+1}$, $w=1$, $t=\alpha_1 p$ and $p=\frac{1+\left(\alpha_1+\sqrt{\alpha_{1}^2+1}\right)^2}{\alpha_1^2+1}$, we obtain all elements of $S_{10}$ with $\alpha_2=1$ and $\alpha_1^2+1\neq0$;
		
		\item By applying $\phi$ in $\mathcal Q_4$ with $x=1, y=0, z=\pm i, w=1, t=0, p=0$, we obtain the elements of $S_{10}$ with $\alpha_2=1$ and $\alpha_1^2+1=0$;
	\end{enumerate}

	\item $\theta_2=\nabla_2, \theta_3=\nabla_1, \theta_4=\nabla_3+\nabla_4$
	
	\begin{enumerate}
		\item The element of $S_{11}$ with $\alpha_5=0$ is $\mathcal Q_1$.
		
		\item By applying $ \phi$ in $\mathcal Q_4$ with $x=0, y=-1, z=1, w=0, t=0, p=\alpha_3$, we obtain all elements of $S_{11}$ with $\alpha_5=1$.
	\end{enumerate}

\end{enumerate}

	All cases have been considered and the lemma is proved.
\end{proof}

\subsection{$5$-dimensional central extensions of $\mathfrak{N}_1$}
 There is only one $5$-dimensional central extension defined by 
 \[\langle \nabla_1, \nabla_2, \nabla_3, \nabla_4,\nabla_5 \rangle.\]


\subsection{Classification theorem}
Summarizing all results regarding to classification of distinct orbits, 
we have the classification of all central extensions of the algebra $\mathfrak{N}_1.$
Note that, we are interested only in non-trivial central extensions,
which are non-split and cannot be considered as a central extension of an algebra of smaller dimension than $\mathfrak{N}_1.$

\begin{theorem}
Let $[\mathfrak{N}_1]^i$ be an $i$-dimensional non-trivial central extension of the Zinbiel algebra $\mathfrak{N}_1.$ Then $[\mathfrak{N}_1]^i$ is isomorphic to one algebra from the following list:

\begin{enumerate}
    \item if $i=1:$
    \begin{longtable}{lllllll}

$[\mathfrak{N}_1]^1_{01}$ &$:$& $e_1e_2=  e_3$&$e_1e_3=e_4$
&$ e_2e_1= -e_3$ &$e_2e_2=e_4$ 
\\

$[\mathfrak{N}_1]^1_{02}$ &$:$& $e_1e_2=  e_3$&$e_1e_3=e_4$
&$ e_2e_1= -e_3$ \\

    \end{longtable}
 
    \item if $i=2:$
    \begin{longtable}{llllllllll}

$[\mathfrak{N}_1]^2_{01}$ &$:$& $e_1e_1=e_4$& $e_1e_2=  e_3$&$e_1e_3=e_5$  &$ e_2e_1= -e_3$ \\

$[\mathfrak{N}_1]^2_{02}$ &$:$& $e_1e_2=  e_3+e_4$&$e_1e_3=e_5$  &$ e_2e_1= -e_3$ \\

$[\mathfrak{N}_1]^2_{03}$ &$:$& $e_1e_2=  e_3$  &$ e_2e_1= -e_3$ &$e_1e_3=e_5$&$e_2e_2=e_4$\\

$[\mathfrak{N}_1]^2_{04}$ &$:$&$e_1e_1=e_4$& $e_1e_2=  e_3$  &$e_1e_3=e_5$&$ e_2e_1= -e_3$ &$e_2e_2=e_4$\\

$[\mathfrak{N}_1]^2_{05}$ &$:$& $e_1e_2=  e_3$ &$e_1e_3=e_5$ &$ e_2e_1= -e_3$ &$e_2e_3=e_4$ \\

$[\mathfrak{N}_1]^2_{06}$ &$:$& $e_1e_1=e_4$&$e_1e_2=  e_3$ &$e_1e_3=e_5$ &$ e_2e_1= -e_3$ &$e_2e_3=e_4$\\

$[\mathfrak{N}_1]^2_{07}$ &$:$& $e_1e_2=  e_3+e_4$ &$e_1e_3=e_5$ &$ e_2e_1= -e_3$&$e_2e_3=e_4$ \\

$[\mathfrak{N}_1]^2_{08}$ &$:$&$e_1e_1=e_4$& $e_1e_2=  e_3$ &$e_1e_3=e_5$ &$ e_2e_1= -e_3$ &$e_2e_2=e_4$&$e_2e_3=e_4$\\

$[\mathfrak{N}_1]^2_{09}$ &$:$&$e_1e_1=e_4$& $e_1e_2=  e_3$ &$e_1e_3=e_5$ &$ e_2e_1= -e_3$&$e_2e_2=e_5$ \\

$[\mathfrak{N}_1]^2_{10}$ &$:$&  $e_1e_2=  e_3+e_4$ &$e_1e_3=e_5$ &$ e_2e_1= -e_3$&$e_2e_2=e_5$

    \end{longtable}

    \item if $i=3:$
    \begin{longtable}{llllllllll}

$[\mathfrak{N}_1]^3_{01}$ &$:$  
&$e_1e_1=e_5$& $e_1e_2=  e_3+e_4$ &$e_1e_3=e_6$&$ e_2e_1= -e_3$ \\

$[\mathfrak{N}_1]^3_{02}$ &$:$  
&$e_1e_1=e_5$& $e_1e_2=  e_3$ &$e_1e_3=e_6$&$ e_2e_1= -e_3$ &$e_2e_2=e_4$\\

$[\mathfrak{N}_1]^3_{03}$ &$:$  
&$e_1e_1=e_5$& $e_1e_2=  e_3$ &$e_1e_3=e_6$&$ e_2e_1= -e_3$ &$e_2e_3=e_4$ \\

$[\mathfrak{N}_1]^3_{04}$ &$:$  
&$e_1e_1=e_5$& $e_1e_2=  e_3$ &$e_1e_3=e_6$&$ e_2e_1= -e_3$ &$e_2e_2=e_4$&$e_2e_3=e_4$ \\

$[\mathfrak{N}_1]^3_{05}$ &$:$  &
$e_1e_2=  e_3+e_5$&$e_1e_3=e_6$ &$ e_2e_1= -e_3$ &$e_2e_2=e_4$ \\

$[\mathfrak{N}_1]^3_{06}$ &$:$  &$e_1e_1=e_4$&
$e_1e_2=  e_3+e_5$&$e_1e_3=e_6$ &$ e_2e_1= -e_3$ &$e_2e_2=e_4$ \\

$[\mathfrak{N}_1]^3_{07}$ &$:$  &
$e_1e_2=  e_3+e_5$&$e_1e_3=e_6$ &$ e_2e_1= -e_3$ &$e_2e_3=e_4$ \\

$[\mathfrak{N}_1]^3_{08}$ &$:$  &$e_1e_1=e_4$&
$e_1e_2=  e_3+e_5$&$e_1e_3=e_6$ &$ e_2e_1= -e_3$ &$e_2e_3=e_4$ \\

$[\mathfrak{N}_1]^3_{09}$ &$:$  &$e_1e_1=e_4$&
$e_1e_2=  e_3$ &$e_1e_3=e_6$&$ e_2e_1= -e_3$&$e_2e_2=e_5$ &$e_2e_3=e_4$\\

$[\mathfrak{N}_1]^3_{10}$ &$:$  &$e_1e_1=e_4+e_5$&
$e_1e_2=  e_3$ &$e_1e_3=e_6$&$ e_2e_1= -e_3$&$e_2e_2=e_5$&$e_2e_3=e_4$ \\

$[\mathfrak{N}_1]^3_{11}$ &$:$  &$e_1e_1=e_5$&
$e_1e_2=  e_3+e_4$&$e_1e_3=e_6$ &$ e_2e_1= -e_3$&$e_2e_2=e_6$ \\
 
    \end{longtable}

    \item if $i=4:$
    \begin{longtable}{llllllllll}

$[\mathfrak{N}_1]^4_{01}$ &$:$  
& $e_1 e_1=e_6$ &$e_1e_2=  e_3+e_5$&$e_1 e_3=e_7$ &$ e_2e_1= -e_3$ &$e_2 e_2=e_4$\\

$[\mathfrak{N}_1]^4_{02}$ &$:$  
&$e_1 e_1=e_6$ & $e_1e_2=  e_3+e_5$&$e_1 e_3=e_7$ &$ e_2e_1= -e_3$  &$e_2 e_3=e_4$ \\

$[\mathfrak{N}_1]^4_{03}$ &$:$  
&$e_1 e_1=e_6$ & $e_1e_2=  e_3$ &$e_1 e_3=e_7$ &$ e_2e_1= -e_3$ &$e_2 e_2=e_5$ &$e_2 e_3=e_4$ \\

$[\mathfrak{N}_1]^4_{04}$ &$:$  
&$e_1 e_1=e_4$ & $e_1e_2=  e_3+e_6$ &$e_1 e_3=e_7$ &$ e_2e_1= -e_3$&$e_2 e_2=e_5$ &$e_2 e_3=e_4$

    \end{longtable}
 
    \item if $i=5:$
    \begin{longtable}{lllllllllll}

$[\mathfrak{N}_1]^5_{01}$ &$:$&$e_1e_1=e_4$& 
$e_1e_2=  e_3+e_5$ &$e_1e_3=e_7$ &$ e_2e_1= -e_3$  &$e_2e_2=e_6$&$e_2e_3=e_8$\\

    \end{longtable}
 
\end{enumerate}
\end{theorem}

\end{document}